\documentclass[12pt]{amsart}

\usepackage[english]{babel}
\usepackage[T1]{fontenc}
\usepackage[utf8]{inputenc} 

\usepackage{amssymb,latexsym}
\usepackage{mathtools}
\usepackage{enumerate}  
\usepackage{a4wide}    

\usepackage{hyperref}
\usepackage{cleveref}
\usepackage{latexsym,amsmath,amssymb}
\usepackage{accents}
\usepackage{color}  

\title[Calderon-Zygmund theory for nonlocal equations]{Calderon-Zygmund theory for non-convolution type nonlocal equations with continuous coefficient}
\author[M.M. Fall]{Mouhamed Moustapha Fall}
\address[Mouhammed Moustapha Fall]{African Institute for Mathematical Sciences in Senegal, AIMS-Senegal  KM 2, Route de Joal  B.P. 14 18 Mbour, Sénégal  }
\email{mouhamed.m.fall@aims-senegal.org}
\author[T. Mengesha]{Tadele Mengesha}
\address[Tadele Mengesha]{Department of Mathematics, The University of Tennessee, Knoxville, 204 Ayres Hall, 1403 Circle Drive
Knoxville, TN, 37996.
}
\email{mengesha@utk.edu}

\author[A. Schikorra]{Armin Schikorra}
\address[Armin Schikorra]{Department of Mathematics,
University of Pittsburgh,
301 Thackeray Hall,
Pittsburgh, PA 15260, USA}
\email{armin@pitt.edu}

\author[S. Yeepo]{Sasikarn Yeepo}
\address[Sasikarn Yeepo]{Department of Mathematics and Computer Science, Faculty of Science, Chulalongkorn University, Bangkok 10330, Thailand}
\email{6072857023@student.chula.ac.th}

plus 6pt minus 12pt
\def\eps{\varepsilon}

\def\F{{\mathcal F}}

\def\N{{\mathbb N}}

\def\S{{\mathbb S}}

\def\T{{\mathbb T}}

\newcommand{\subsubset}{\subset\subset}

\newtheorem{theorem}{Theorem}

\newtheorem{lemma}[theorem]{Lemma}
\newtheorem{corollary}[theorem]{Corollary}
\newtheorem{proposition}[theorem]{Proposition}

\theoremstyle{definition}

\def\supp{{\rm supp\,}}

\newcommand{\R}{\mathbb{R}}

\newcommand{\Z}{\mathbb{Z}}

\newcommand{\brac}[1]{\left (#1 \right )}
\newcommand{\abs}[1]{\left |#1 \right |}

\newcommand{\barint}{
\rule[.036in]{.12in}{.009in}\kern-.16in \displaystyle\int }

\newcommand{\barcal}{\mbox{$ \rule[.036in]{.11in}{.007in}\kern-.128in\int $}}

\def\mvint_#1{\mathchoice
          {\mathop{\vrule width 6pt height 3 pt depth -2.5pt
                  \kern -8pt \intop}\nolimits_{\kern -3pt #1}}%
          {\mathop{\vrule width 5pt height 3 pt depth -2.6pt
                  \kern -6pt \intop}\nolimits_{#1}}%
          {\mathop{\vrule width 5pt height 3 pt depth -2.6pt
                  \kern -6pt \intop}\nolimits_{#1}}%
          {\mathop{\vrule width 5pt height 3 pt depth -2.6pt
                  \kern -6pt \intop}\nolimits_{#1}}}

\numberwithin{theorem}{section} \numberwithin{equation}{section}

\newcommand{\lap}{\Delta }

\newcommand{\aleq}{\lesssim}
\newcommand{\ageq}{\gtrsim}
\newcommand{\aeq}{\approx}

\newcommand{\laps}[1]{(-\lap) ^{\frac{#1}{2}}}

\newcommand{\lapms}[1]{I^{#1}}

\begin{document}
\begin{abstract}
Given $2\leq p<\infty$, $s\in (0, 1)$ and $t\in (1, 2s)$, we establish interior $W^{t,p}$ Calderon-Zygmund estimates for solutions of nonlocal equations of the form
\[
 \int_{\Omega} \int_{\Omega} K\left (x,|x-y|,\frac{x-y}{|x-y|}\right ) \frac{(u(x)-u(y))(\varphi(x)-\varphi(y))}{|x-y|^{n+2s}} dx dy = g[\varphi], \quad \forall \phi\in C_c^{\infty}(\Omega)
\]
where $\Omega\subset \mathbb{R}^{n}$ is an open set. 
 Here we assume $K$ is bounded, nonnegative and continuous in the first entry -- and ellipticity is ensured by assuming that $K$ is strictly positive in a cone.

The setup is chosen so that it is applicable for nonlocal equations on manifolds, but the structure of the equation is general enough that it also applies to the certain fractional $p$-Laplace equations around points where $u \in C^1$ and $|\nabla u| \neq 0$.
\end{abstract}

\everymath{\displaystyle}
\maketitle
\tableofcontents
\section{Introduction and main results}
Let $s \in (0,1)$  and $\Omega \subset \R^n$ be an open set. In this work we study the interior $W^{t,p}$-regularity theory for functions $u\in W^{s, 2}(\Omega)$ that satisfy the nonlocal equation
{\small \begin{equation}\label{eq:ourpde}
 \int_{\Omega} \int_{\Omega}  K\brac{x,|x-y|,\frac{x-y}{|x-y|}} \frac{{(u(x)-u(y))(\varphi(x)-\varphi(y))}}{|x-y|^{n+2s}}\, dx\, dy = f[\varphi], \quad \forall \varphi \in C_c^\infty(\Omega).
\end{equation}}
\noindent where $f[\varphi]$ represents the action of the distribution $f$ on the smooth function $\varphi$. 
We take $K(\cdot,\cdot,\cdot): \R^n \times [0,\infty) \times \S^{n-1} \to [0,\Lambda]$ as our coefficient kernel and is a nonnegative bounded and measurable function which is strictly positive on a substantial subset of $\S^{n-1}$. The model case is $K(x,r,h) = \frac{|\phi(x)-\phi(x+rh)|}{|r|}$ for a map $\phi$ with $|\nabla \phi(x)| > 0$. In particular, if $\phi$ is a diffeomorphism that parameterizes an $n$-dimension manifold $\mathcal{M}$, then \eqref{eq:ourpde} could be a nonlocal differential equation on $\mathcal{M}$. Other choices of naturally appearing kernels will be studied later. 

For $t \geq s$, local $W^{t,p}$-Calderon Zygmund theory has been studied in \cite{MSYHoelder20} for a different type of kernels by three of the authors. On the other hand the first author recently studied H\"older regularity theory for \eqref{eq:ourpde} in \cite{Fall18}. To some extent, with respect to motivation, techniques, and results, this paper is a combination of the ideas in \cite{Fall18} and \cite{MSYHoelder20}, and we refer to those papers for a background and motivation of equations of the type we are considering.

The following theorem states the main result of the paper. 
\begin{theorem}\label{th:main}
Suppose that $ \Omega \subseteq \R^n$ is open and $\Omega_1 \subsubset \Omega$ is bounded.  Assume $\lambda > 1$, $\eta>0$, $s_1+s_2 = 2s$ and $s_1,s_2,s\in(0,1)$, $p \in (1,\infty)$.

Let $K:\R^n \times \R_+ \times \S^{n-1} \to \R$ be in all entries bounded and measurable, assume that the kernel is nondegenerate in the following sense
\[
 0 \leq K(x,r,h) \leq \Lambda \quad \forall x \in \R^n, r \geq 0,\ h \in \S^{n-1} 
\]
and there exists    $\Sigma \subset \S^{n-1}$ with $|\Sigma|>0$ such that for every $x\in \mathbb{R}^{n}$
\begin{equation}\label{eq:lambdaangleass}
  \eta\leq  \inf_{h \in \Sigma} K(x,0,h).
\end{equation}
Assume also the following continuity around points on $\Omega_1$: for every $x_0\in \Omega_1 $,  and some $\mathfrak{R}>0$ such that $B(x_0, 5\mathfrak{R})\subsubset \Omega$, and each $\eps > 0$,  there exists $\lambda> 0$ such that
\begin{equation}\label{eq:smallness}
 \sup_{r <\lambda \mathfrak{R}} \sup_{x \in B(x_0,\lambda \mathfrak{R})} \sup_{|h| = 1} |K(x,r,h)-K(x_0,0,h)| \leq \eps.
\end{equation}
Now, for a given a given distribution $g$, let $u \in W^{s,2}(\Omega)$ satisfies
\[
 \int_{\Omega}\int_{\Omega} K\brac{x,|x-y|,\frac{x-y}{|x-y|}} \frac{(u(x)-u(y))(\varphi(x)-\varphi(y))}{|x-y|^{n+2s}}\, dx\, dy = g[\varphi], \quad \forall \varphi \in C_c^\infty(\Omega_1). 
\]
Then, if for some $\Lambda > 0$,  we have the estimate 
\[
 |g[\varphi]| \leq \Lambda\, \left(\|\varphi\|_{L^{p'}(\R^n)}+[\varphi]_{W^{s_2,p'}(\R^n)}\right), \]
for all $\varphi \in C_c^\infty(\Omega_1)$,    
then  $u \in W^{s_1,p}_{loc}(\Omega_1)$. \end{theorem}
In the statement above, $p'$ is the H\"older conjugate of $p$. Notice that the theorem is interesting when $s_2 < s$. 
For $p=2$ the result of the theorem is only slightly better than \cite{cozzi} where H\"older continuity of $K$ is assumed, but for $ s_1\neq s$ and $p \neq 2$ this is a new result. Recently, S. Nowak developed $W^{s,p}$-regularity theory under the weaker vanishing mean oscillation assumptions on the kernel, \cite{Nowak19,nowak2020higher,nowak2021}. It is a natural question if for our main theorem the continuity condition (for $t < 1$) can be relaxed to a VMO-type condition, but we will leave this issue for a future work.


We conjecture that the condition $s_1 < 1$ in \Cref{th:main} to be sharp, although we do not have a counterexample. However, if we additionally assume H\"older continuity of the kernel, the result of \Cref{th:main} holds for $s_1 > 1$ as well. The following theorem states this result. 

\begin{theorem}\label{th:main:above1}
Assume all the conditions of \Cref{th:main} hold. Assume additionally $\alpha$-H\"older continuity of $K$ in the first variable, i.e.
\begin{equation}\label{eq:Khoelder}
 \sup_{x,y \in \R^n} \sup_{|h| = 1, r\geq 0} |K(x,r,h)-K(y,r,h)| \leq \Lambda |x-y|^\alpha.
\end{equation}
Then the conclusion of \Cref{th:main} hold for $s_1 < \min\{2s,1+\alpha\}$ (instead of $s_1 < \min\{2s,1\}$).
\end{theorem}

%
%
We refer to \cite{Fall18} for applicability of our result to the fractional mean curvature, cf. \cite{Fall18nonlocmcvsurvey}. The structure of the kernel $K$ appears naturally in applications that leads to linear nonlocal differential equations on $n$-manifolds $\mathcal{M}$, say 
\[
 \mathcal{L}_{\mathcal{M},s}(u,\varphi) := \int_{\mathcal{M}} \int_{\mathcal{M}} \frac{\brac{u(x)-u(y)}\brac{\varphi(x)-\varphi(y)}}{d_{\mathcal{M}}(x,y)^{n+2s}}\, dx\, dy.
\]
for smooth $\phi \in C_{0}^{\infty}(\mathcal{M}).$
For simplicity, we may assume that all of $\mathcal{M}$ is parametrized by a bi-lipschitz diffeomorphism $\Phi: \Omega \to \mathcal{M}$ (otherwise one can work on a coordinate patch with the help of a decomposition of unity). Then for any $\psi \in C_c^\infty(\Omega)$, by change of variables we have 
\[
\mathcal{L}_{\mathcal{M},s}(u,\psi \circ \Phi^{-1}) = \int_{\Omega} \int_{\Omega} \frac{\brac{\tilde{u}(x)-\tilde{u}(y)}\brac{\psi(x)-\psi(y)}} {d_{\mathcal{M}}(\Phi(x),\Phi(y))^{n+2s}}\, {\rm Jac}(D\Phi(x))\, {\rm Jac}(D\Phi(y))\, dx\, dy 
\]
where $\tilde{u} := u \circ \Phi$. 
Thus,  if we set 
\[
 K(x,r,h) := \brac{\frac{r}{d_{\mathcal{M}} (\Phi(x),\Phi(x+rh))}}^{n+2s} {\rm Jac}(D\Phi(x))\, {\rm Jac}(D\Phi(x+rh))
\]
we have 
\[
 \mathcal{L}_{K,s}(\tilde{u},\psi) = \mathcal{L}_{\mathcal{M},s}(u,\psi \circ \Phi^{-1}).
\]
Observe that 
\[
  K(x,0,h) \aeq |D\Phi(x) h|^{-n-2s} 
\]
which is uniformly bounded from below because $D\Phi$ has maximal rank as a diffeomorphism. In particular we have the following

\begin{corollary}\label{co:main:diffeo}
The conclusion of \Cref{th:main} is true for the equation 
\[
 \int_{\Omega}\int_{\Omega} \frac{(u(x)-u(y))(\varphi(x)-\varphi(y))}{|\phi(x)-\phi(y)|^{n+2s}}\, dx\, dy = g[\varphi] \quad \forall \varphi \in C_c^\infty(\Omega)
\]
where $\phi: \Omega \subset \R^n \to \R^m$ is bi-Lipschitz and $C^{1,\alpha}$-differentiable.
\end{corollary}

Lastly, let us illustrate a consequence of our arguments for the higher regularity theory of the fractional $p$-Laplacian in the regime where $|\nabla u(x_0)| > 0$ for $s \approx 1$.
\begin{corollary}\label{co:plap}
Let $s \in (0,1)$, $p \geq 2$ such that $s > \frac{p-2}{p}$. For a given $\Omega \subseteq \R^n$ open set, let $u \in W^{s,p}\cap L^\infty (\Omega)$ be a solution to
{\small \begin{equation}\label{eq:plapeq}
 \int_{\Omega} \int_{\Omega} \frac{|u(x)-u(y)|^{p-2} (u(x)-u(y))\, (\varphi(x)-\varphi(y))}{|x-y|^{n+sp}}\, dx\, dy = \int_{\Omega} f \varphi dx \quad \forall \varphi \in C_c^\infty(\Omega). 
\end{equation}}
If $f \in C_c^\infty(\R^n)$,  $u$ is continuously differentiable around some point $x_0 \in \Omega$, and $|\nabla u(x_0)| > 0$,  then there exists a small neighborhood $B(x_0,r)$ such that $u \in C^{2}(B(x_0,r))$.
\end{corollary}
Let us stress a subtlety of \Cref{co:plap}. The idea here is to take the solution $u$ as part of the kernel $ K(x, r, h) = \frac{|u(x)-u(x+rh)|^{p-2}}{|h|^{s(p-2)}}$.  Even if we assume that $|\nabla u(x_0)| \neq 0$, we still might have \[
\frac{|u(x_0)-u(x_0+rh)|}{r} \approx |\nabla u(x_0)\cdot h|\approx 0, 
\]
that for sufficiently small $r>0$, for example, whenever $h$ is perpendicular to $\nabla u(x_0)$. However, we know that for a substantial subset $\Sigma \subset \S^{n-1}$, that depends on $x_0$, we have for all $h \in \Sigma$ that $|\nabla u(x_0) h| \ageq |\nabla u(x_0)|$, namely this is true for all $h \in S^{n-1}$ with $|\angle (h,\nabla u(x_0))| > \lambda > 0$. Condition \eqref{eq:lambdaangleass} on the kernel in \Cref{th:main} will allow us to apply the theorem to handle such cases. The details of \Cref{co:plap} are explained in \Cref{s:co:plap}.

Let us stress that \Cref{co:plap} is an expected result but likely not optimal. The a priori $L^\infty$-condition on $u$ can be waived, indeed in view of \cite{CastroKuusiPalatucciJFA14} we can assume that $u$ is even (H\"older-) continuous. 
Also, it may be possible to remove the conditions $p \geq 2$ and $s>1-\frac{2}{p}$ in \Cref{co:plap} and as well obtain differentiability for the solution above $C^2$. The crucial condition $|\nabla u(x_0)| > 0$ makes the equation ``somewhat uniformly elliptic'' with coefficients which regularize with $u$. All these are the crucial ingredients to bootstrap towards $C^\infty$-regularity -- however, there is some subtlety here that we do not address in this paper: differentiating the equation even for $p>2$ the fractional $p$-Laplacian becomes singular elliptic (even around points where $|\nabla u(x_0)| > 0$). 

Another point to keep in mind with respect to \Cref{co:plap} is that the more challenging case is around points $x_0$ with vanishing derivative, $|\nabla u(x_0)| = 0$ and the (here a priori assumed) $C^1$-regularity cannot be obtained by \Cref{th:main}. Perhaps other techniques such as those developed by Brasco--Lindgren \cite{BL17,BLS18} may be needed. 

The statement of \Cref{co:plap} is also somewhat motivated by the recent extention formula for the (sublinear, $p<2$) $p$-Laplacian obtained in \cite{dTGCV20} -- where they assumed $u \in C^2$ and $|\nabla u(x_0)| > 0$. While our arguments do not immediately apply for $p < 2$ we hope that our techniques have some use for this equation to remove the (very strong) $C^2$-assumption.

{We conclude the introduction by describing the organization of the paper as well as the approach we follow to prove the main result. We use a perturbation argument to obtain the interior regularity estimate stated in  \Cref{th:main}. We compare the nonlocal equation of interest locally with a nonlocal equation on the $n$-torus whose solutions are periodic functions with better regularity. Results pertaining to regularity theory to solutions to nonlocal equation on the $n$-torus are presented in \Cref{s:hongjidong}.}

\subsection*{Acknowledgments}
M.M. F. is funded  by the Alexander von Humboldt foundation. T.M. is funded by National Science Foundation (NSF), grant no 1910180. A.S. is funded by the NSF Career award DMS-2044898 and Simons foundation grant no 1910180.

\section{The torus and periodic maps - Dong-Kim's regularity result}\label{s:hongjidong}

{To facilitate the perturbation argument we use to obtain the main result,  we develop 
regularity theory for nonlocal PDE with convolution-type kernel on the periodic domain the $n$-torus which is the cube $\bigg[{-\frac{1}{2}},{\frac{1}{2}}\bigg]^n$ with opposite sides identified. The regularity results we obtain in this section are a reformulation of the results obtained in \cite{DONG20121166, Mikulevicius2014}  where a regularity theory is developed for nonlocal equations over $\R^n$. }


First we fix some notations. The $n$-torus $\T^n=\bigg[{-\frac{1}{2}},{\frac{1}{2}}\bigg]^n$ is given as the set of equivalence class $\R^n / \sim$ where the equivalence relation $\sim$ on $\R^n$ defined as follows: we say that, for $x,y \in \R^n$, $x \sim y$ if $x-y \in \Z^n$. With this identification, periodic functions on $\T^n$ are characterized as $f$ defined on $\R^n$ and satisfy $f(x+m) = f(x)$ for all $x \in \R^n$ and $m \in \Z^n$. The space of infinitely differentiable functions on $\T^n$ denoted by $C^\infty(\T^n)$. For $1\leq p \leq \infty$, $f$ is in $L^{p}(\T^n)$ if $\|f\|_{L^{p}(\T^n)} < \infty. $ 
For a complex-valued function $f$ in $L^1(\T^n)$, we define the Fourier transform of $f$, $\mathcal{F}(f)$,  and for $m \in \Z^n$  as 
\[
\mathcal{F}(f)(m) = \int_{\T^n} f(x) e^{-2\pi \imath\langle m, x\rangle}\, dx, 
\]
where $\langle m, x\rangle = \sum_{i=1}^{n}m_ix_i$ is the dot product of $m$ and $x$. 
The inverse Fourier transform (also known as Fourier series of $f$) at $x \in \T^n$ is the series
\[
 \sum_{m \in \Z^n} \mathcal{F}(f)(m) e^{2\pi \imath\langle m, x\rangle}.
\] 
We can define other operators on function spaces of $\T^n$-periodic functions. 
For example, for $s>0$, we define the fractional Laplacian on $\T^n$, $\Delta_{\T^n}^{s/2}$,   via the Fourier series as 
\[
\Delta_{\T^n}^{{\frac{s}{2}}}f(x) = \sum_{m\in\Z^n}|m|^{s} \mathcal{F}(f)(m) e^{2\pi \imath\langle m,x \rangle}. 
\]
See \cite{RS16} for more on the fractional Laplacian on $\T^n$. 

For $1<p<\infty$ and $s>0$ we define the inhomogeneous Sobolev and Bessel potential spaces on the torus are defined, respectively, as 
\[
W^{s,p}(\T^n)=\left\{u\in L^{p}(\T^n): |u|_{W^{s,p}(\T^n)}^{p}=\int_{\T^n}\int_{\T^n} {\frac{|u(x)-u(y)|}{|x-y|^{d+ps}}} dxdy < \infty\right\}
\]
and
\[
H^{s,p}(\T^n)=\left\{u\in L^{p}(\T^n):  |u|_{\dot{H}^{s,p}(\T^n)}= \|\Delta_{\T^n}^{{\frac{s}{2}}}u\|_{L^{p}(\T^n)} < \infty \right\}.
\]
For $p=2$, $W^{s,2}(\T^n) = H^{s,2}(\T^n)$, which can be shown via Plancheral Theorem. We denote $H^{s,2}(\T^n)$ just simply by $H^s(\T^n)$. The homogeneous Bessel potential space $\dot{H}^{s,p}(\T^n)$ on the torus collects all periodic distributions $u$ such that $|u|_{\dot{H}^{s,p}(\T^n)}= \|\Delta_{\T^n}^{{\frac{s}{2}}}u\|_{L^{p}(\T^n)} < \infty$.   
We should mention the Sobolev inequality for potential spaces: for $s>0$ and $1<p<q<\infty$ such that ${\frac{s}{n}} \geq  {\frac{1}{p}} - {\frac{1}{q}}$, we have 
\[
\|u-(u)_{\T^n}\|_{L^{q}(\T^n)} \leq C | u|_{\dot{H}^{s, p}(\T^n)},
\]
where $C= C(s, n, p, q)>0$ and $(\cdot)_{\T^n}$ represents the average operator, see \cite{Benyi-Oh2013}. 
%
%
%
%


The main result of this section is stated in the following theorem which obtains a Calderon-Zygmond type result for nonlocal equations on the torus with kernels that are homogeneous of degree zero.  We denote by $\S^{n-1} \subset \R^n$ the $n-1$ sphere. 
\begin{theorem}\label{th:trinv}
Let  $\eta>0$, $s_1+s_2 = 2s$, $s_2 < s$, and $s_1,s_2,s\in(0,1)$. Suppose that $p \in (1,\infty)$ and $\Sigma \subset\S^{n-1}$ such that $|\Sigma|>0$. 
Let $K : \S^{n-1} \to [0,\Lambda]$ be a nonnegative, bounded and measurable function, such that 
\begin{equation}\label{eq:anglecond}
0< \eta := \inf_{\Sigma} K.
\end{equation}
Then there exists a constant $C = C(\eta,s_1,s_2,s,\Sigma)>0$ such that the following holds.

Then for any  $g$ with $\|g\|_{W^{-s_2,p}(\R^n)} \equiv \|g\|_{\brac{W^{s_2,p'}(\R^n)}^\ast}  \aleq \Lambda$, i.e.
$
 |g[\varphi]| \leq \Lambda\, [\varphi]_{W^{s_2,p'}(\R^n)},
$
and a solution $u \in W^{s,2}(\T^n)$ to the equation
\[
 \int_{\T^n}\int_{\T^n} \mu\brac{x-y} (u(x)-u(y))(\varphi(x)-\varphi(y))\, dx\, dy = g[\varphi] \quad \forall \varphi \in C^\infty(\T^n)
\]
with 
\[
 \mu(h) := \sum_{m \in \Z^n} \frac{K(\frac{h+m}{|h+m|})}{|h+m|^{n+2s}}, 
\]
we have  $u\in W^{s_1,p}(\T^n)$. Moreover,  we have the estimate
\[
 [u]_{W^{s_1,p}(\T^n)} \leq C\, \Lambda+\|u\|_{L^2(\T^n)}.
\]
\end{theorem}
Our proof of the theorem follows that of  a similar result proved by H. Dong and D. Kim in \cite{DONG20121166} for nonlocal equations posed on $\mathbb{R}^{n}$. There are some distinctions however with our setup no matter how minor they are. First, the nonlocal equation we study is posed over the torus. Second, we have a relaxed ellipticity condition, namely,  ellipticity of $K$ is assumed over subset of  $\mathbb{S}^{d}$ with positive measure as opposed to the whole of  $\mathbb{S}^{d}$. As the following lemma demonstrates, this is enough to establish the ellipticity of the associated operator. 
\begin{lemma}\label{la:ftintegralcond-Cone}
Let $K: \S^{n-1} \to [0,1]$ be measurable and satisfy \eqref{eq:anglecond} with $\Sigma \subset \S^{n-1}$ and $|\Sigma|>0$.
Then, for $R,R'>0$, with $RR'\geq \overline r>0$ there exists $c=c(n,s,\Sigma,\overline r)>0$ such that  for all $\xi \in \R^N\setminus B (0, R')$, 
\begin{equation} \label{eq:low-Cone-BR}
\int_{B(0, R)}\frac{1-\cos(2\xi\cdot h)}{|h|^{n+2s}}K(h/|h|)\, dh\geq c\, \eta |\xi|^{2s}.
\end{equation}
If $R=+\infty$ and $R'=0$, then \eqref{eq:low-Cone-BR} holds with $\overline r=1$.
In particular, for any $k \in \Z^n$, we have
$$
 \int_{\T^n} (1-\cos(2\pi \langle k, h\rangle )) \sum_{m\in \Z^n}\frac{K\left(\frac{h+m}{|h+m|}\right)}{|h+m|^{n+2s}}\, dh \geq  c(n,s,\Sigma) \eta |k|^{2s}.
$$
\end{lemma}
\begin{proof}
By \eqref{eq:anglecond}, for $\xi \in \R^N\setminus B (0, R')$, we have 
\begin{align}
\int_{B(0, R)}\frac{1-\cos(2\xi\cdot h)}{|h|^{n+2s}}K(h/|h|)\, dh&= \int_{\S^{N-1}}\int_0^R\frac{1-\cos(2\xi\cdot t\theta)}{t^{1+2s}}K(\theta)\, d\theta dt \nonumber\\
&= |\xi|^{2s}\int_{\S^{N-1}}\int_0^{|\xi| R}\frac{1-\cos(2\xi/|\xi|\cdot r \theta)}{r^{1+2s}}K(\theta)\, d\theta dr \nonumber\\
&\geq  \eta  |\xi|^{2s}\int_{\Sigma}\int_0^{R'R}\frac{1-\cos(2\xi/|\xi|\cdot r \theta)}{r^{1+2s}} \, d\theta dr \nonumber\\
&\geq  \eta  |\xi|^{2s}\int_{\Sigma}\int_0^{\overline r}\frac{1-\cos(2\xi/|\xi|\cdot r \theta)}{r^{1+2s}} \, d\theta dr. \label{eq:low-Cone}
\end{align}
From the Taylor expansion of cosine, we can find a  constant $r_0\in (0, \overline r)$ such that 
 \begin{equation}\label{eq:1-cos}
1-\cos(2\sigma \cdot r \theta)\geq r^2 \frac{|\sigma \cdot  \theta|^2}{4} \qquad\textrm{ for all $\sigma,\theta\in \S^{n-1}$ and $r\in (0, r_0)$}.
\end{equation}
Moreover, it is clear that   the map 
$$
 \S^{n-1}\to \R, \qquad \sigma \mapsto f(\sigma):= \int_{\Sigma}  {|\sigma \cdot  \theta|^2}  \, d\theta 
$$
attains its  minimum at some $\sigma_0\in \S^{n-1}$. It is plain that $f(\sigma_0)>0$, because otherwise  $\sigma \cdot  \theta=0$ for all $\theta\in \Sigma$, which contradicts that $|\Sigma|>0$.
Using this and \eqref{eq:1-cos}  in \eqref{eq:low-Cone}, we obtain
\begin{align*}
\int_{B(0,R)}\frac{1-\cos(2\xi\cdot h)}{|h|^{n+2s}}K(h/|h|)\, dh 
&\geq  \eta  \frac{|\xi|^{2s}}{4}\int_{\Sigma}\int_0^{r_0}\frac{|\xi/|\xi|\cdot  \theta|^2}{r^{-1+2s}} \, d\theta dr  \geq \frac{ \eta r_0^{2-2s}f(\sigma_0)}{8(1-s)} |\xi|^{2s}.
\end{align*}
This gives \eqref{eq:low-Cone-BR}.
\end{proof}

\begin{lemma}
With $K$ and $\mu$ as in \Cref{th:trinv}, let \[
	Lu(x) := \int_{\T^n} (u(x) - u(y))\,\mu(x-y)\, dy.
	\]
	Let $u \in W^{s,2}(\T^n)$ satisfing $Lu = g$ in the distributional sense in $\T^n$, where $g \in L^2(\T^n)$. Then 
	\[
	\|u\|_{H^{2s}(\T^n)} \aleq \|g\|_{L^2(\T^n)}.
	\]
\end{lemma}
\begin{proof}
	First, we consider the Fourier transform of $Lu$ on torus. Observe that for $u \in W^{s,2}(\T^n)$ the expression $Lu$ is a distibution on $\T^n$ - in particular $\F\, Lu(k)$ is well-defined since $e^{-2\pi \imath \langle k,x \rangle} \in C^\infty(\T^n)$. 
	Then
	{\small \[
	\begin{split}
	\F(Lu)(k) &= \left \langle Lu(x),e^{-2\pi \imath \langle k,x \rangle}\right \rangle= \frac{1}{2} \int_{\T^n} \int_{\T^n} (u(x) - u(y))\,\mu(x-y) \brac{e^{-2\pi \imath \langle k,x \rangle}-e^{-2\pi \imath \langle k,y \rangle}}\, dy\,dx.
	\end{split}
	\]}
	Changing variable $x-y = h$, we obtain
	\[
	\begin{split}
	\F(Lu)(k) &= \frac{1}{2}\int_{\T^n} \int_{\T^n} (u(x) - u(x-h))\,\mu(h)\, \brac{e^{-2\pi \imath \langle k,x \rangle}-e^{-2\pi \imath \langle k,x-h \rangle}}\, dx\,dh\\
&=	\int_{\T^n} \int_{\T^n} (u(x) - u(x-h))\,\mu(h)\, e^{-2\pi \imath \langle k,x \rangle}\, dx\,dh\\
	&=\F(u)(k) \int_{\T^n} (1-e^{2\pi \imath \langle k,h \rangle})\,\mu(h)\,dh.
	\end{split}
	\]
	By Plancherel's identity, we get
	\[
	\begin{split}
	\|Lu\|_{L^2(\R^n)}^2 &= \sum_{k \in \Z^n} |\F(Lu)(k)|^2\\
	&= \sum_{k \in \Z^n} |\F(u)(k)|^2 \left|\int_{\T^n} (1-e^{2\pi \imath \langle k,h \rangle})\,\mu(h)\,dh\right|^2 \\
	&\ge  \sum_{k \in \Z^n} |\F(u)(k)|^2 \left|\Re\int_{\T^n} (1-e^{2\pi \imath \langle k,h \rangle})\,\mu(h)\,dh\right|^2 \\
	&= \sum_{k \in \Z^n} |\F(u)(k)|^2 \left(\int_{\T^n} (1-\cos (2\pi \langle k,h \rangle))\,\mu(h)\,dh\right)^2 \\
	&= \sum_{k \in \Z^n} |\F(u)(k)|^2 \left(\int_{\T^n} (1-\cos (2\pi \langle k,h \rangle))\,\sum_{m \in \Z^n} \frac{K({\frac{h+m}{|h+m|}})}{|h+m|^{n+2s}}\,dh\right)^2 .
	%
	\end{split}
	\]
	In view of \Cref{la:ftintegralcond-Cone}, we find that 
	\[
	\|g\|_{L^2(\T^n)}^2 =  
	 \|Lu\|_{L^2(\T^n)}^2 \ageq \sum_k ||k|^{2s} \F(u)(k)|^2 = \|u\|_{H^{2s}(\T^n)}^2.
	\]
%
%

%
\end{proof}

With these above lemmata we obtain by almost verbatim arguments as in \cite{DONG20121166} the following result. As for the H\"older regularity result, we sketch in \Cref{s:hoeldercont} an alternative argument.
\begin{proposition} \label{pr:hd} Let $u,\mu,g$ be as in \Cref{th:trinv}. Then,
	\[
	\|\laps{2s}_{\T^n} u\|_{L^p(\T^n)} \aleq \|g\|_{L^p(\T^n)}.
	\]
\end{proposition}

We also observe that we can differentiate the convolution type equation. Namely, we have

\begin{lemma}\label{la:differentiating} Let $u,\mu,g$ be as in \Cref{th:trinv}. 
	Assume that $u$ satisfies
	\[
	\int_{\T^n}\int_{\T^n} (u(x)-u(y))(\varphi(x)-\varphi(y))\, \mu\brac{x-y}dx\, dy = g[\varphi] \quad \forall \varphi \in C_{per}^\infty(\T^n).
	\]
	Then $v := \laps{\sigma} u$, for $\sigma \in \R$, satisfies
	\[
	\int_{\T^n}\int_{\T^n} (v(x)-v(y))(\varphi(x)-\varphi(y))\, \mu\brac{x-y}dx\, dy = g[\laps{\sigma}\varphi] \quad \forall \varphi \in C_{per}^\infty(\T^n).
	\]
\end{lemma}
\begin{proof}
	For every fixed $y \in \T^n$, by changing of variable choosing $h=x-y$, we get
	\[
	\begin{split}
	&\int_{\T^n}\int_{\T^n} (v(x)-v(y))(\varphi(x)-\varphi(y))\, \mu\brac{x-y}dx\, dy\\
	&= \int_{\T^n}\int_{\T^n} (\laps{\sigma} u(x)-\laps{\sigma} u(y))(\varphi(x)-\varphi(y))\, \mu\brac{x-y}dx\, dy\\
	&= \int_{\T^n}\int_{\T^n} (\laps{\sigma} u(y+h)-\laps{\sigma} u(y))(\varphi(y+h)-\varphi(y))\, \mu\brac{h}dh\, dy.
	\end{split}
	\]
	Then, using Parseval’s relation, we get
	\[
	\begin{split}
	&= \int_{\T^n} \sum_{m \in \Z^n} |e^{2\pi\imath m\cdot h}-1|^2 |m|^\sigma\hat{u}(m)\hat{\varphi}(m) \,\mu\brac{h} \,dh\\
	&= \int_{\T^n} \sum_{m \in \Z^n} |e^{2\pi\imath m\cdot h}-1|^2 \hat{u}(m)\widehat{\laps{\sigma}\varphi}(m) \,\mu\brac{h} \,dh\\
	&= \int_{\T^n}\int_{\T^n} (u(y+h)-u(y))(\laps{\sigma} \varphi(y+h)-\laps{\sigma} \varphi(y))\, \mu\brac{h}dh\, dy\\
	&= \int_{\T^n}\int_{\T^n}  (u(x)-u(y))(\laps{\sigma} \varphi(x)-\laps{\sigma} \varphi(y))\, \mu\brac{x-y}dx\, dy\\
	&= g[\laps{\sigma}\varphi].
	\end{split}
	\]
\end{proof}

Now we obtain the full result of \Cref{th:trinv} from interpolation.

\begin{proof}[Proof of \Cref{th:trinv}]
Let $T$ be the solution operator, i.e., for given $g$, $Tg$ is the solution of
\[
 \int_{\T^n}\int_{\T^n} (Tg(x)-Tg(y))(\varphi(x)-\varphi(y))\, \mu\brac{x-y}dx\, dy = g[\varphi] \quad \forall \varphi \in C^\infty(\T^n).
\]
 Observe that $Tg = u$ where such a solution can be shown to exist via variational methods.
From \Cref{pr:hd}, we obtain
\[
\|\laps{2s}Tg\|_{L^p(\T^n)} \aleq \|g\|_{L^p(\T^n)}.
\]
In view of \Cref{la:differentiating} the above estimate implies that for any $\sigma \in \R$ 
%
%
\[
\|\laps{2s+\sigma}Tg\|_{L^p(\T^n)} \aleq \|\laps{\sigma}g\|_{L^p(\T^n)}.
\]
That is for any $\sigma \in \R$, 
\[
T: \dot{H}^{\sigma,p}(\T^n) \mapsto \dot{H}^{2s+\sigma,p}(\T^n)
\]
is a linear, bounded operator. 
Take $\sigma_0 := -s$ and $\sigma_1 := 0$. Observe that $H^{\sigma,p} = F^{\sigma}_{p,2}$.
Then we have
\[
T: \dot{H}^{-s,p}(\T^n) \mapsto \dot{H}^{s,p}(\T^n),
\]
\[
T: \dot{H}^{s,p}(\T^n) \mapsto \dot{H}^{2s,p}(\T^n)
\]
are both bounded operators.

By \cite[2.5.1,p.86, Proposition and Remark]{RS96} we obtain by interpolation 
\[
T: B^{\sigma}_{p,p}(\T^n) \mapsto B^{2s+\sigma}_{p,p}(\T^n)
\]
is continuous and bounded for any $\sigma \in (-s,0)$, that is
\[
\|Tg\|_{B^{2s+\sigma}_{p,p}(\T^n)} \aleq \|g\|_{B^{\sigma}_{p,p}(\T^n)}.
\]
Since $B^{\sigma}_{p,p} = W^{\sigma,p}$ we conclude for $\sigma := -s_2$, 
\[
\|Tg\|_{W^{s_1,p}(\T^n)}  = \|Tg\|_{B^{s_1}_{p,p}(\T^n)} \aleq \|g\|_{B^{-s_2}_{p,p}(\T^n)} \equiv \|f\|_{\brac{W^{s_2,p'}(\T^n)}^\ast}.
\]
Since $u=Tg$ we conclude.
\end{proof}

\section{Local estimates}
Our proof of \Cref{th:main} uses a perturbative argument where we view the nonlocal equation of interest as a small perturbation of a convolution-type equation on the torrus. This is achieved by a freezing of coefficient method facilitated by the uniform H\"older assumption on the kernel $K$ given by  \eqref{eq:smallness}. This leads to a convolution-type equation as treated in \cite{DONG20121166}, see \Cref{th:trinv}. 
Applying \Cref{th:trinv} on the convolution-type equation leads to local priori estimates (because of the term \eqref{eq:Fest}). We convert these a priori estimates to regularity estimates with the help of a fixed point argument in \Cref{s:fixedpoint}. 

To that end, in this section for a given $\Omega\subset \mathbb{R}^{n}$ and $\Omega'\subset\subset \Omega$, we fix $u \in W^{s,2}(\R^n) \cap W^{t,p}(\Omega')$ that satisfies the equation 
{\small \begin{equation}\label{eq:freezepde}
 \int_{\Omega}\int_{\Omega} K\brac{x,|x-y|,\frac{x-y}{|x-y|}}\, \frac{\brac{u(x)-u(y)}\brac{\varphi(x)-\varphi(y)}}{|x-y|^{n+2s}} dx\, dy = g[\varphi] \quad \forall \varphi \in C_c^\infty(\Omega'), 
\end{equation} }
for a given distribution $g.$ 
Let also $x_0 \in \Omega'$ and $R>0$ such that $B(x_0,60\sqrt{n}R) \subset \Omega'$. Set the $n$-cube $\tilde{\T} := x_0+30R [-\frac{1}{2},\frac{1}{2}]^n$ which is a subset of $B(x_0,60\sqrt{n}R)$.  
Without loss of generality we may assume that $x_0=0$ and $R=\frac{1}{30}$ so that $0\in \Omega'\subset\subset \Omega$, $\tilde{\T} $ becomes the cube $[-\frac{1}{2},\frac{1}{2}]^n$, and $B(0, 2\sqrt{n})$ contains a number of copies of $\tilde{\T}.$
Define the periodic function 
\[
 \mu(h) := \sum_{k \in \Z^n} \frac{K(0,0,\frac{h+k}{|h+k|})}{|h+k|^{n+2s}}, \quad  \text{for $h\in \tilde{\T}$}. 
 \]
Let $\eta \in C_c^\infty(B(0,{\frac{1}{5}}))$ with $\eta \equiv 1$ in $B(0,{\frac{1}{6}})$. 
\subsection{Local {\em a priori} estimates via freezing coefficients argument }

\begin{proposition}[Freezing the coefficient]\label{pr:rewrite}

Let $s \in (0,1)$, $t \in [s,1)$ and $p \in [2,\infty)$. Let $\Omega \subset \R^n$ and $\Omega' \subset \subset \Omega$. 
Let $u \in W^{s,2}(\R^n) \cap W^{t,p}(\Omega')$ be a solution to
{\small \[\tag{\ref{eq:freezepde}}
 \int_{\Omega}\int_{\Omega} K\brac{x,|x-y|,\frac{x-y}{|x-y|}}\, \frac{\brac{u(x)-u(y)}\brac{\varphi(x)-\varphi(y)}}{|x-y|^{n+2s}} dx\, dy = g[\varphi] \quad \forall \varphi \in C_c^\infty(\Omega').
\] }
where $K$ is bounded, nondegenerate, and satisfies \Cref{eq:lambdaangleass}. 
Let $x_0 \in \Omega'$ and $R>0$ such that $B(x_0,60\sqrt{n}R) \subset \Omega'$, and $\eta \in C_c^\infty(B(x_0,6R))$ with $\eta \equiv 1$ in $B(x_0,5R)$.  Set $\tilde{\T} := x_0+30R \bigg[\frac{-1}{2},\frac{1}{2}\bigg]^n\subset B(x_0,60\sqrt{n}R)$ and define the $\tilde{\T}$-periodic function 
\[
 \mu(h) := \sum_{k \in 30R\Z^n} \frac{K(x_0,0,\frac{h+k}{|h+k|})}{|h+k|^{n+2s}}, \quad  \text{for $h\in \tilde{\T}$}. 
 \]
Denote by $v := \eta u$. Then
\begin{equation}\label{eq:newpde}
 \int_{\tilde{\T}}\int_{\tilde{\T}} \mu(x-y)\, \brac{v(x)-v(y)}\brac{\psi(x)-\psi(y)} dx\, dy = g[\eta\psi] + \mathcal{H}(v,\psi) + \mathcal{G}(u,\psi), 
\end{equation}
for all $\psi \in C^\infty_{per}\brac{\overline{\tilde{\T}}},$ where $\mathcal{H}$ and $\mathcal{G}$ are bilinear forms with the property that 

for any $s_1,s_2 \in (0,1)$ with $s_1+s_2 = 2s$ and any $p_1,p_2 \in (1,\infty)$ with $\frac{1}{p_1}+\frac{1}{p_2}=1$ we have 
\begin{equation}\label{eq:Fest}
  |\mathcal{H}(a,b)| \leq \sup_{x \in B(x_0,10R),|h| = 1, r \in [0,10R]} \abs{K(x,r,h)-K(x_0,0,h)}\  [a]_{W^{s_1,p_1}_{per}(\tilde{\T})}\, [b]_{W^{s_2,p_2}_{per}(\tilde{\T})},
\end{equation}
Moreover for any $\max\{0,2s-1\}< \tilde{t} < 2s-t$ and $q \leq \frac{p}{p-1}$ such that 
\begin{equation}\label{eq:rewrite:tpq}
 \min\left \{t-\frac{n}{q}, \tilde{t}-\frac{n}{q}\right \} \geq -\frac{n}{p'}
\end{equation}

we have 
{\small $
 |\mathcal{G}(a,b)| \leq C(R,\eta,\|K\|_{L^\infty})\, \brac{\|b\|_{L^q(\T^n)} + [b]_{W^{\tilde{t},q}(\T)}} \brac{\|a\|_{L^2(\tilde{\T})} + [a]_{W^{s,2}(\tilde{\T})} + [a]_{W^{t,p}(\tilde{\T})}}.
$}
\end{proposition}
\begin{proof}
Without loss of generality (otherwise we rescale and translate) we may assume that $x_0=0$ and $R=\frac{1}{30}$ so that $0\in \Omega'\subset\subset \Omega$, $\tilde{\T} $ becomes the unit torus $\T^n=\bigg[\frac{-1}{2},\frac{1}{2}\bigg]^n$, and $B(0, 2\sqrt{n})$ contains a number of copies of $\T^n.$ The function $\mu$ now becomes 
\[
 \mu(h) = \sum_{k \in \Z^n} \frac{K(0,0,\frac{h+k}{|h+k|})}{|h+k|^{n+2s}}, \quad  \text{for $h\in \T^n$}, 
 \]
and  $\eta \in C_c^\infty(B(0,{\frac{1}{5}}))$ with $\eta \equiv 1$ in $B(0,{\frac{1}{6}})$. 

For $\psi \in C^\infty_{per}(\overline{\T^n})$, by a slight abuse of notation, we can assume that $\psi$ is equal to its periodic extension $\psi \in C^\infty(\R^n)$.

After writing $\mu$ as 
\[
 \mu(x-y) = \frac{K(0,0,\frac{x-y}{|x-y|})}{|x-y|^{n+2s}} + {\frac{\tilde{K}(x, y)}{|x-y|^{n+2s} }}
\]
with $ \tilde{K}(x, y)= \sum_{k \in \Z^n\backslash \{0\}} K\brac{0,0,\frac{x-y+k}{|x-y+k|}}\frac{|x-y|^{n+2s}}{|x-y+k|^{n+2s}},$
we view it as a perturbation of $\frac{K(0,0,\frac{x-y}{|x-y|})}{|x-y|^{n+2s}}$. This is motivated by the fact that 
{\small\begin{equation}\label{eq:Leftover-mu}
 \sup_{x \in B(0,\frac{1}{5}), y\in  \T^n} \sum_{k \in  \Z^n\backslash \{0\}} \frac{K\brac{0,0,\frac{x-y+k}{|x-y+k|}}}{|x-y+k|^{n+2s}} + \sup_{y \in B(0,\frac{1}{5}), x\in  \T^n} \sum_{k \in  \Z^n\backslash \{0\}} \frac{K\brac{0,0,\frac{x-y+k}{|x-y+k|}}}{|x-y+k|^{n+2s}}  \aleq 1.
\end{equation}}
and thus for $(x,y) \in \T^n \times \T^n$ and either $x\in \supp v$ or $y \in \supp v$ (or both) then
\[
 \abs{\mu(x-y) - \frac{K(0,0,{\frac{x-y}{|x-y|}})}{|x-y|^{n+2s}}} \aleq 1.
\]
Consequently we obtain the decomposition 
\begin{equation}\label{eq:involving-eta}
\begin{split}
  &\int_{\T^n}\int_{\T^n} \mu(x-y)\, \brac{v(x)-v(y)}\brac{\psi(x)-\psi(y)} dx\, dy \\
  =&\int_{\T^n}\int_{\T^n} K\brac{0,0,\frac{x-y}{|x-y|}}\, \frac{\brac{v(x)-v(y)}\brac{\psi(x)-\psi(y)}}{|x-y|^{n+2s}} dx\, dy + \mathcal{G}_{1}(u,\psi) \\
  \end{split}
\end{equation}
where
\[
 \mathcal{G}_{1}(u,\psi) =\int_{\T^n}\int_{\T^n} \tilde{K}(x, y)\, \frac{\brac{v(x)-v(y)}\brac{\psi(x)-\psi(y)}}{|x-y|^{n+2s}} dx\, dy. 
\]

We are going to view the first term in the right hand side of \eqref{eq:involving-eta} as a  'frozen coefficient' of the original nonlocal equation. To make proper comparison,  we set $\tilde{\psi}:=\tilde{\eta} (\psi - (\psi)_{\T^n})$ for a reasonable cutoff function $\tilde{\eta} \in C_c^\infty(2\T^n)$ and $\tilde{\eta} \equiv 1$ in $\T^n$. Notice that $2\T^n\subset\Omega.$
It then follows that 
\[
\begin{split}
 &\int_{\T^n}\int_{\T^n} K\brac{0,0,\frac{x-y}{|x-y|}}\, \frac{\brac{v(x)-v(y)}\brac{\psi(x)-\psi(y)}}{|x-y|^{n+2s}} dx\, dy \\
 =&\int_{\T^n}\int_{\T^n} K\brac{0,0,\frac{x-y}{|x-y|}}\, \frac{\brac{v(x)-v(y)}\brac{\tilde{\psi}(x)-\tilde{\psi}(y)}}{|x-y|^{n+2s}} dx\, dy \\
 =&\int_{\T^n}\int_{\T^n} K\brac{x,|x-y|,\frac{x-y}{|x-y|}}\, \frac{\brac{v(x)-v(y)}\brac{\tilde{\psi}(x)-\tilde{\psi}(y)}}{|x-y|^{n+2s}} dx\, dy \\
 &+\int_{\T^n}\int_{\T^n} E_{K}(x,y)\, \frac{\brac{v(x)-v(y)}\brac{\tilde{\psi}(x)-\tilde{\psi}(y)}}{|x-y|^{n+2s}} dx\, dy ,
 %
 \end{split}
 \]
 where $E_{K}(x, y) = {K\brac{0,0,\frac{x-y}{|x-y|}}-K\brac{x,|x-y|,\frac{x-y}{|x-y|}}}$ is the measure of the pointwise deviation of the main kernel from $K\brac{0,0,\frac{x-y}{|x-y|}}$. 
 We may thus write  
 \[
 \begin{split}
 \int_{\T^n}\int_{\T^n} K\brac{0,0,\frac{x-y}{|x-y|}}&\, \frac{\brac{v(x)-v(y)} \brac{\psi(x)-\psi(y)}}{|x-y|^{n+2s}} dx\, dy \\
 =&\mathcal{H}_1(u,\tilde{\psi}) + \mathcal{H}_2(v,\psi)   + \mathcal{G}_2(u,\psi) + \mathcal{G}_3(u,\tilde{\psi}) + \mathcal{G}_4(u,\tilde{\psi}), 
 \end{split}
 \]
 where $\mathcal{H}_1(u,\tilde{\psi})$ 
 \[
 \mathcal{H}_1(u,\tilde{\psi}) :=\int_{\Omega}\int_{\Omega} K\brac{x,|x-y|,\frac{x-y}{|x-y|}}\, \frac{\brac{v(x)-v(y)}\brac{\tilde{\psi}(x)-\tilde{\psi}(y)}}{|x-y|^{n+2s}} dx\, dy
 \]
and the remaining terms comes from decomposing  the deviation from the frozen term and are given by 
{ \[
 \mathcal{H}_2(v,\psi) := \int_{B(x_0, {\frac{1}{3}})}\int_{B(x_0, {\frac{1}{3}})} E_{K}(x,y)
 \frac{\brac{v(x)-v(y)}\brac{\tilde{\psi}(x)-\tilde{\psi}(y)}}{|x-y|^{n+2s}} dx\, dy,
 \]}
{\small  \[
 \mathcal{G}_2(u,\psi):=\int_{B(x_0, {\frac{1}{3}})}\int_{\Omega \backslash \T^n} K\brac{x, |x-y|,\frac{x-y}{|x-y|}}\, \frac{\brac{\eta(x)u(x)-\eta(y)u(y)}\brac{\tilde{\psi}(x)-\tilde{\psi}(y)}}{|x-y|^{n+2s}} dx\, dy, 
 \]}
{\small  \[
 \mathcal{G}_3(u,\psi) := \int_{\Omega \backslash \T^n} \int_{B(x_0, {\frac{1}{3}})}K\brac{x, |x-y|,\frac{x-y}{|x-y|}}\, \frac{\brac{\eta(x)u(x)-\eta(y)u(y)}\brac{\tilde{\psi}(x)-\tilde{\psi}(y)}}{|x-y|^{n+2s}} dx\, dy,  
 \]}
 and after recalling that the support of $\eta$ is in $B(0, {\frac{1}{5}})$)
{  \[
 \begin{split}
 \mathcal{G}_4(u,\psi)
& :=\int_{{B(0, {\frac{1}{5}})}}\int_{\T^n \backslash B(0,{\frac{1}{3}})} E_{K}(x, y)
  \frac{\brac{\eta(x)u(x)-\eta(y)u(y)}\brac{\tilde{\psi}(x)-\tilde{\psi}(y)}}{|x-y|^{n+2s}} dx\, dy \\
 &+\int_{\T^n \backslash B(0,{\frac{1}{3}})} \int_{B(x_0,{\frac{1}{5}})} E_{K}(x, y)
  \frac{\brac{\eta(x)u(x)-\eta(y)u(y)}\brac{\tilde{\psi}(x)-\tilde{\psi}(y)}}{|x-y|^{n+2s}} dx\, dy \\
 \end{split}
\]}


Using the identity $(a-b)(cd-ef)+(c-e)(bd-af) = (ac-be)(d-f)$ we may rewrite the expression in the integrand of $\mathcal{H}_1$ as 
\[
\begin{split}
 &(v(x)-v(y))(\tilde{\psi}(x)-\tilde{\psi}(y))\\
 &=(u(x)-u(y))(\eta(x)\tilde{\psi}(x)-\eta(y)\tilde{\psi}(y))
 +(\eta(x)-\eta(y))  \brac{u(y)\tilde{\psi}(x)-u(x) \tilde{\psi}(y)}. 
 \end{split}
\]
Consequently, using the fact that $u$ solves the nonlocal equation \eqref{eq:freezepde} we have 
\[
 \mathcal{H}_1(u,\tilde{\psi}) = g[\eta \tilde{\psi}] + \mathcal{G}_5(u,\psi), 
\]
where
\[
 \mathcal{G}_5(u,\psi) := \int_{\Omega}\int_{\Omega} K\brac{x,|x-y|,\frac{x-y}{|x-y|}}\frac{\brac{\eta(x)-\eta(y)} \brac{u(y)\tilde{\psi}(x)-u(x) \tilde{\psi}(y)} }{|x-y|^{n+2s}} dx\, dy. 
\]
We may now summarize the above discussion to rewrite \eqref{eq:involving-eta} as   
\[
\int_{\T^n}\int_{\T^n} \mu(x-y)\, \brac{v(x)-v(y)}\brac{\psi(x)-\psi(y)} dx\, dy=g[\eta \tilde{\psi}]  + \mathcal{H}(v, \psi) + \mathcal{G}(u, \psi),
\] 
where $\mathcal{H}(v, \psi) = \mathcal{H}_2(v, \psi)$, and  $\mathcal{G}(u, \psi) = \sum_{i=1}^{5} \mathcal{G}_{i}(u, \psi)$, which is precisely the expression in \eqref{eq:newpde}.  What remains is to give estimate to each of the terms to conclude the proof of the proposition. To that end, 
the estimate \eqref{eq:Fest} follows from the expression $\mathcal{H}_2(v,\psi)$ and from H\"older's inequality.
For the estimate for $\mathcal{G}$ will estimate each of the terms of $\mathcal{G}$. We begin estimating $\mathcal{G}_1$. First observe that using \eqref{eq:Leftover-mu} and since $v= \eta u\in W^{t, p}_{0}(\T^n)$ applying  Poincar\'e inequality to obtain 
\[
\begin{split}
 \mathcal{G}_{1}(u,\psi) &\aleq \int_{\T^n}\int_{\T^n} |v(x)-v(y)||\psi(x)-\psi(y)|\, dx\, dy\\
& \aleq \|u\|_{W^{t,p}(\T^n)}\, \brac{\int_{\T^n}\int_{\T^n} |\psi(x)-\psi(y)|^{p'}}^{\frac{1}{p'}}. 
\end{split}
\]
By Sobolev-Poincar\`e embedding the last factor can be estimates as 
{ \[
 \brac{\int_{\T^n}\int_{\T^n} |\psi(x)-\psi(y)|^{p'}\, dx\, dy}^{\frac{1}{p'}} \aleq [\psi]_{W^{\tilde{t},q}(\T^n)}, 
\]}
so that we have 
\[
 \mathcal{G}_{1}(u,\psi)  \aleq \|u\|_{W^{t,p}(\T^n)}\, [\psi]_{W^{\tilde{t}, q}(\T^n)}.
\]

To estimate of $\mathcal{G}_2$, we notice that for $x \in \Omega \backslash \T^n$ and $y \in B(0,{\frac{1}{3}})$ we have $|x-y| \aeq |x-y|+1$. 
Thus,
\[
\begin{split}
 |\mathcal{G}_2(u,\psi)|\aleq &\int_{B(0, {\frac{1}{3}})}\int_{\Omega \backslash \T^n} \frac{1}{|x-y|^{n+2s}+1} \abs{u(y)} \abs{\tilde{\psi}(x)-\tilde{\psi}(y)} dx\, dy\\
 \aleq& \|\tilde{\psi}\|_{L^1(\Omega)}\, \|u\|_{L^1(B(0, {\frac{1}{3}})} + \|u\tilde{\psi}\|_{L^1(B(0, {\frac{1}{3}}))}.
 \end{split}
 \]
Observe that from periodicity and Sobolev embedding we obtain that 
\[
 \|\tilde{\psi}\|_{L^{1}(\R^n)} + \|\tilde{\psi}\|_{L^{p'}(\R^n)} \aleq [\psi]_{W^{\tilde{t},q}(\T^n)}, 
\]
and 
\[
 \|u\|_{L^{1}(B(0, {\frac{1}{3}}))}+\|u\|_{L^{p}(B(0, {\frac{1}{3}}))} \aleq \brac{\|u\|_{L^2(\T)} + [u]_{W^{s,2}(\T)} + [u]_{W^{t,p}(\T)}}. 
\]
Thus we obtain the estimate 
\[
 \mathcal{G}_2(u,\psi) \aleq [\psi]_{W^{\tilde{t},\tilde{q}}(\T^n)}\, \brac{\|u\|_{L^2(\T)} + [u]_{W^{s,2}(\T)} + [u]_{W^{t,p}(\T)}}. 
\]
The procedure to estimate $\mathcal{G}_3$ is exactly the same as that of $\mathcal{G}_2$ and so we have 
\[
\begin{split}
 \mathcal{G}_3(u,\psi) \aleq &\int_{\Omega \backslash \T^n} \int_{B(0, {\frac{1}{3}})} \frac{1}{|x-y|^{n+2s}+1}\ |u(x)|\abs{\tilde{\psi}(x)-\tilde{\psi}(y)} dx\, dy \\
 \aleq&[\psi]_{W^{\tilde{t},\tilde{q}}(\T^n)}\, \brac{\|u\|_{L^2(\T)} + [u]_{W^{s,2}(\T)} + [u]_{W^{t,p}(\T)}}.
 \end{split}
 \]
To estimate of $\mathcal{G}_4$ first notice that since the support of $\eta$ is contained in $B(0, {\frac{1}{5}})$, 
\[
 |\mathcal{G}_4(u,\psi)|=\abs{\int_{B(0, {\frac{1}{5}})}\int_{\T^n \backslash B(0, {\frac{1}{3}})} \frac{\abs{-\eta(y)u(y)\, \brac{\tilde{\psi}(x)-\tilde{\psi}(y)}}}{|x-y|^{n+2s}} dx\, dy }\]
Now arguing in exactly similar way as the estimate for $\mathcal{G}_2$ we obtain 
\[
\begin{split}
 |\mathcal{G}_4(u,\psi)|
 \aleq &\int_{B(0, {\frac{1}{5}}))}\int_{\T^n \backslash B(0, {\frac{1}{3}})} \frac{1}{|x-y|^{n+2s}+1} |u(y)|\, \abs{\tilde{\psi}(x)-\tilde{\psi}(y)} dx\, dy \\
 \aleq&[\psi]_{W^{\tilde{t},q}(\T^n)}\, \brac{\|u\|_{L^2(\T)} + [u]_{W^{s,2}(\T)} + [u]_{W^{t,p}(\T)}}.
 \end{split}
\]

Finally, we estimate $\mathcal{G}_5$ as follows: using $|\eta(x)-\eta(y)|\aleq |x-y|$ for $x, y\in \Omega'$, 
\[
\begin{split}
 |\mathcal{G}_5(u,\tilde{\psi})| \aleq& \int_{\Omega'}\int_{\Omega'} \frac{|u(x)| |\tilde{\psi}(x)-\tilde{\psi}(y)| + |\tilde{\psi}(x)| |u(x)-u(y)|}{|x-y|^{n+2s-1}} dx\, dy\\
 &+\int_{\Omega \backslash \Omega'} \int_{\Omega'} \chi_{|x-y| \ageq R}\frac{|u(x)| |\tilde{\psi}(y)|+ |\tilde{\psi}(x)| |u(y)|}{|x-y|^{n+2s}}\, dx\, dy. 
\end{split}
\]
For the first term, we observe that 
\[
\begin{split}
\int_{\Omega'}\int_{\Omega'} \frac{|u(x)| |\tilde{\psi}(x)-\tilde{\psi}(y)| }{|x-y|^{n+2s-1}} dx\, dy
=&\int_{\Omega'}\int_{\Omega'} |x-y|^{\tilde{t}-(2s-1)} |u(x)|  \frac{|\tilde{\psi}(x)-\tilde{\psi}(y)| 
}{|x-y|^{n+\tilde{t}}} dx\, dy\\
\aleq &[\tilde{\psi}]_{W^{\tilde{t},q}}\, \brac{\int_{\Omega'}\int_{\Omega'} |x-y|^{\tilde{t}q'-(2s-1)q'-n} |u(x)|^{q'}  \, dx\,dy}^{\frac{1}{q'}}\\
\aleq&[\tilde{\psi}]_{W^{\tilde{t},q}}\, \|u\|_{L^{q'}(\Omega')}\\
\aleq&[\tilde{\psi}]_{W^{\tilde{t},q}}\, \brac{\|u\|_{L^{2}(\Omega')}+[u]_{W^{t,p}(\Omega')}}.
\end{split}
\]
The the second inequality follows from the fact that $\tilde{t}>2s-1$ and 
\[
\sup_{x\in \Omega'} \int_{\Omega'} |x-y|^{\tilde{t}q'-(2s-1)q'-n}dy \leq C(\text{diam}(\Omega'), q',n,s,\tilde{t} ).
\]

Also since $t > s$ and so $t> 2s-1$ and $s \in (0,1)$, similarly as above we have 
\[
\begin{split}
\int_{\Omega'}\int_{\Omega'} \frac{|\tilde{\psi}(x)| |u(x)-u(y)|}{|x-y|^{n+2s-1}} dx\, dy
\aleq [u]_{W^{t,p}(\Omega')}\, \|\tilde{\psi}\|_{L^{p'}(\Omega')}
\aleq [u]_{W^{t,p}(\Omega')}\, [\psi]_{W^{\tilde{t},p'}(\T^n)}.
\end{split}
\]

As for the second term: as in \cite[Lemma 5.2.]{MSYHoelder20} for any $r \in [1,\infty]$,
\[
 \sup_{x \in \Omega'} \int_{\Omega \backslash \Omega'}  \chi_{|x-y| \ageq R} \frac{|g(y)|}{|x-y|^{n+2s}}\, dy \aleq \|g\|_{L^r(\Omega)}
\]
 
and thus
\[
 \begin{split}
\left |\int_{\Omega \backslash \Omega'} \int_{\Omega'} \chi_{|x-y| \ageq R}\frac{|u(x)| |\tilde{\psi}(y)|+ |\tilde{\psi}(x)| |u(y)|}{|x-y|^{n+2s}}\, dx\, dy \right | 
\aleq& \|\tilde{\psi}\|_{L^1(\Omega')}\, \|u\|_{L^{2}(\Omega)} + \|u\|_{L^1(\Omega')}\, \|\tilde{\psi}\|_{L^{q}(\Omega)}\\
\aleq& \|\tilde{\psi}\|_{L^q(\Omega')}\, \|u\|_{L^{2}(\Omega)}\\
\aleq&[\psi]_{W^{\tilde{t},q}(\T^n)}\, \|u\|_{L^{2}(\Omega)}.
\end{split}
\]
Collecting the estimates for $\mathcal{G}_{i}$ for $i=1, 2\cdots, 5$ completes the proof of the proposition.  
\end{proof}
\begin{corollary}\label{apEst}
Fix $u \in W^{t,p}\cap W^{s,2}(\Omega')$ under the assumptions for $K, t, \tilde{t}$, $s$, $p$, and $q$ as well as $R$ and $\tilde{\T}$ as in \Cref{pr:rewrite}. 
For a given distribution $g$ such that there is some $\Lambda > 0$ with the 
\[
 |g[\varphi]| \leq \Lambda\, \left(\|\varphi\|_{L^{q}(\R^n)}+[\varphi]_{W^{\tilde{t},q}(\R^n)}\right), \]
for all $\varphi \in C_c^\infty(B(x_0,30R))$,    let $v \in W^{s,2}(\tilde{\T})$ be a solution to  \eqref{eq:newpde}. 
If in addition $v \in W^{2s-\tilde{t},q'}(\tilde{\T})$,  then there exists a small enough $\eps>0$ (depending on $x_0, $$t$, $s$, $p$) such that for any $K$ satisfying \eqref{eq:smallness} we have 
\begin{equation}\label{eq:fix:vest}
[v]_{W^{2s-\tilde{t},q'}(\tilde{\T})} \aleq \|u\|_{L^2(\tilde{\T})} + [u]_{W^{s,2}(\tilde{\T})} + [u]_{W^{t,p}(\tilde{\T})} + \Lambda.@
\end{equation}

\end{corollary}
\begin{proof}
We notice that \eqref{eq:newpde} can be rewritten as 
\[
 \int_{\tilde{\T}}\int_{\tilde{\T}} \mu(x-y)\, \brac{v(x)-v(y)}\brac{\psi(x)-\psi(y)} dx\, dy =\Upsilon[\psi] ,
\]
where $\Upsilon[\psi]:=g[\eta\psi] + \mathcal{H}(v, \psi) + \mathcal{G}(u, \psi)$ for any $\psi \in C^{\infty}_{per}(\tilde{\T}). $ 
Next, under all the assumptions of \Cref{pr:rewrite}, the  estimates for the bilinear forms $\mathcal{H}$ and $\mathcal{G}$ combined with \eqref{eq:smallness}  we have the estimate for 
\[
|\Upsilon[\psi]| \aleq [\psi]_{W^{\tilde{t},q}(\tilde{\T})}\left( \eps [v]_{W^{2s-\tilde{t},q'}(\tilde{\T})}+\|u\|_{L^2(\tilde{\T})} + [u]_{W^{s,2}(\tilde{\T})} + [u]_{W^{t,p}(\tilde{\T})} + \Lambda\right).
\]
We can now apply \Cref{th:trinv}  and obtain
\[
 [v]_{W^{2s-\tilde{t},q'}(\tilde{\T})}  \aleq \eps [v]_{W^{2s-\tilde{t},q'}(\tilde{\T})} + \|u\|_{L^2(\tilde{\T})} + [u]_{W^{s,2}(\tilde{\T})} + [u]_{W^{t,p}(\tilde{\T})}  + \Lambda.
\]
Since by assumption $[v]_{W^{2s-\tilde{t},\tilde{q}'}(\tilde{\T})} < \infty$ we obtain \eqref{eq:fix:vest} by absorption of the $\eps [v]_{W^{2s-\tilde{t},\tilde{q}'}(\T^n)}$ term after choosing $\eps$ small enough.
\end{proof}
\subsection{Local improved regularity  via a fixed point theorem}\label{s:fixedpoint}
We remark that the estimate given in \Cref{apEst} is an {\em a priori} estimate because it relies on the regularity assumption $v\in W^{2s-\tilde{t},{q}'}(\tilde{\T})$. In this subsection we prove this local regularity result  under the assumption of \Cref{pr:rewrite}
\begin{theorem}\label{th:fixedpoint}
Fix $u \in W^{t,p}\cap W^{s,2}(\Omega')$ under the assumptions for $t, \tilde{t}$, $s$, $p$, and $q$ as in \Cref{pr:rewrite}. Assume that $g$ is a distribution such that there is some $\Lambda > 0$ with the 
\[
 |g[\varphi]| \leq \Lambda\, \left(\|\varphi\|_{L^{q}(\R^n)}+[\varphi]_{W^{\tilde{t},q}(\R^n)}\right), \]
for all $\varphi \in C_c^\infty(B(x_0,30R))$. 
Then there exists a small enough $\eps>$ (depending on $x_0, $$t$, $s$, $p$ such that for any $K$ satisfying \eqref{eq:smallness}, 
any solution $v \in W^{s,2}(\tilde{\T})$ to \eqref{eq:newpde} belongs to $W^{2s-\tilde{t},q'}(\tilde{\T})$.
\end{theorem}
%
%
%
\begin{proof}
We prove the theorem in two steps. In the first step we show that any $W^{s,2}(\tilde{\T})$ solution to \eqref{eq:newpde} is unique up a constant. Next, we demonstrate that there is a solution $v \in W^{2s-\tilde{t},q'}\cap W^{s,2}(\tilde{\T})$ to \eqref{eq:newpde}. As a consequence, we conclude that any solution $v \in W^{s,2}(\tilde{\T})$ to \eqref{eq:newpde} belongs to $W^{2s-\tilde{t},q'}(\tilde{\T})$.
%

For the uniqueness, assume there are two solutions $v$ and $\tilde{v}$ in $W^{s,2}(\tilde{\T})$. Set $w := v-\tilde{v}$, then \eqref{eq:newpde} implies (recall that $u$ and $g$ are fixed)
\[
 \int_{\tilde{\T}}\int_{\tilde{\T}} \mu(x-y)\, \brac{w(x)-w(y)}\brac{\psi(x)-\psi(y)} dx\, dy = \mathcal{H}(w,\psi).
\]
Testing (by a density argument) with $\psi=w$ we find in view of \Cref{eq:Fest}
\[
 [w]_{W^{s,2}(\tilde{\T})}^2 \aleq \sup_{x \in B(x_0,10R),|h| = 1, r \leq 20R} \abs{K(x,r,h)-K(x_0,r,h)}\  [w]_{W^{s,2}(\tilde{\T})}^2.
\]
Then choosing $\eps$ sufficiently small, for any $K$ satisfying \eqref{eq:smallness}, it follows that $[w]_{W^{s,2}(\tilde{\T})}^2 = 0$ and thus $w \equiv const$.

Let us now demonstrate that there is a solution $v$ that belongs to $W^{2s-\tilde{t},q'}\cap W^{s,2}(\tilde{\T})$. To that end, let 
\[
           X := \left \{v \in W^{s,2}(\tilde{\T}):\ (v)_{\tilde{\T}}  = 0\right \}.
          \]
By Poincar\'e inequality $X$ endowed with the $[\cdot]_{W^{s,2}}$-seminorm is a Banach space.

Starting from $v_0 := 0$ define $v_{k+1} \in X$  inductively as the minimizer of $\inf_{X} \mathcal{E}_k(\cdot)$ where
\[
\mathcal{E}_k(v) :=  \frac{1}{2} \int_{\tilde{\T}}\int_{\tilde{\T}} \mu(x-y)\abs{v(x)-v(y)}^2 dx\, dy - g[\eta v] - \mathcal{H}(v_k,v) - \mathcal{G}(u,v) .
\]
The minimizer $v_{k+1} \in X$ exists by the direct method in the calculus of variations. The minimizer $v_{k+1}$ satisfies the Euler-Lagrange equation
\[
  \int_{\tilde{\T}}\int_{\tilde{\T}} \mu(x-y) {\brac{v_{k+1}(x)-v_{k+1}(y)}\brac{\psi(x)-\psi(y)}} dx\, dy = g[\eta\psi] + \mathcal{H}(v_k,\psi) + \mathcal{G}(u,\psi), 
\]
for all $\psi \in C^\infty_{per}\brac{\overline{\tilde{\T}}}.$

Denote by $(v_{k+1})_\tau=v_{k+1}\ast \phi_\tau$ the convolution with respect to $\tilde{\T}$ where  $\phi_\tau$ is a approximate identity in $C^{\infty}(\overline{\tilde{\T}})$. 
By the convolution type of the equation for $v_{k+1}$ we have 
\[
\begin{split}
  \int_{\tilde{\T}}\int_{\tilde{\T}}& \mu(x-y) {\brac{(v_{k+1})_{\tau}(x)-(v_{k+1})_\tau (y)}\brac{\psi(x)-\psi(y)}} dx\, dy \\
  &=  \int_{\tilde{\T}}\int_{\tilde{\T}} \mu(x-y) \brac{v_{k+1}(x)-v_{k+1} (y)}\brac{\psi_{\tau}(x)-\psi_{\tau}(y)} dx\, dy\\
  &= g[\eta\psi_\tau] + \mathcal{H}(v_k,\psi_\tau) + \mathcal{G}(u,\psi_\tau) \quad \forall \psi \in C^\infty_{per}\brac{\overline{\tilde{\T}}}.
  \end{split}
\]
We observe that the map
\[
\psi\in C^\infty_{per}\brac{\overline{\tilde{\T}}}\mapsto \Upsilon_{(k,\tau)} (\psi):=g[\eta\psi_\tau] + \mathcal{H}(v_k,\psi_\tau) + \mathcal{G}(u,\psi_\tau) 
\]
satisfies the estimate that 
\[
|\Upsilon_{(k,\tau)} (\psi)|\leq C\|\psi\|_{W^{\tilde{t}, q}} \left( \eps [v_{k}]_{W^{2s-\tilde{t},\tilde{q}'}(\tilde{\T})} + \brac{\|u\|_{L^2(\tilde{\T})} + [u]_{W^{s,2}(\tilde{\T})} + [u]_{W^{t,p}(\tilde{\T})}} +\Lambda\right)
\]
for a given $\epsilon>0$ and $K$ satisfying \eqref{eq:smallness} where the constant $C$ is uniform in $\tau, \epsilon,k$ for any $\epsilon$.
Since $(v_{k+1})_\tau \in C^\infty_{per}(\tilde{\T}) \subset W^{2s-\tilde{t},q'}_{per}(\tilde{\T})$ we can apply \Cref{apEst}. 
and obtain uniform $W^{2s-\tilde{t},q'}_{per}(\tilde{\T})$-estimates for $(v_{k+1})_\tau$, and as $\tau \to 0$ we thus obtain $v_{k+1} \in W^{2s-\tilde{t},q'}_{per}(\tilde{\T})$ with the estimate 
\[
 [v_{k+1}]_{W^{2s-\tilde{t},q'}(\tilde{\T})}  \aleq \eps [v_{k}]_{W^{2s-\tilde{t},q'}(\tilde{\T})} + \brac{\|u\|_{L^2(\tilde{\T})} + [u]_{W^{s,2}(\tilde{\T})} + [u]_{W^{t,p}(\tilde{\T})}} + \Lambda. 
 \]
Repeating this argument for $w_{k+1} := v_{k+1}-v_k$, which satisfies the equation 
\[
  \int_{\tilde{\T}}\int_{\tilde{\T}} \mu(x-y) {\brac{w_{k+1}(x)-w_{k+1} (y)}\brac{\psi(x)-\psi(y)}} dx\, dy = \mathcal{H}(w_k,\psi) \quad \forall \psi \in C^\infty_{per}\brac{\overline{\tilde{\T}}}.
\]
we obtain the estimate (recall \eqref{eq:smallness})
\[
 [w_k]_{W^{2s-\tilde{t},q'}(\tilde{\T})} \leq C\eps [w_{k-1}]_{W^{2s-\tilde{t},q'}(\tilde{\T})}.
\]
Since $(w_k)_{\tilde{\T}} = 0$, this implies that $v_k$ is convergent in $W^{2s-\tilde{t},q'}(\tilde{\T})$, and passing to the limit we find a solution $v \in W^{2s-\tilde{t},q'}(\tilde{\T})$ of \eqref{eq:newpde}.
\item Since $2s-\tilde{t} > s$ and $q'\geq 2$ we have $W^{2s-\tilde{t},q'}(\tilde{\T})\subset W^{s,2}(\tilde{\T})$ and thus by uniqueness (up to a constant) of $v$ for $W^{s,2}$-maps we have that any solution belongs to $W^{2s-\tilde{t},q'}(\tilde{\T})$. That concludes the proof. 

\end{proof}

\section{\texorpdfstring{$W^{t,p}$}{Wtp}-estimates, \texorpdfstring{$t > 1$}{t>1}: differentiating the equation}
Up to now we obtained $W^{t,p}$-estimates for $t < 1$ -- under suitable assumptions on the right-hand side, without assuming anything but continuity of the kernel involved. For $W^{t,p}$-estimates for $t \geq 1$, we need $t-1$-H\"older continuity of the kernel in order to differentiate the equation.

\begin{theorem}\label{th:diffeq}
Let $s \in (0,1)$ and $\alpha \in (0,1)$. Assume $K$ is as in \Cref{th:main}, and additionally we have H\"older continuity in the first variable,
\[
 |K(x,r,h)-K(y,r,h)| \leq \Lambda |x-y|^\alpha \quad \forall x,y \in \R^n, |h| = 1, r \geq 0.
\]
Suppose also that, for $t \in (s,1)$, $p \geq 2$, $u \in W^{t,p}\cap W^{s,2}(\R^n)$ is a solution of
\[
\int_{\R^n} \int_{\R^n}  K\brac{x,|x-y|,\frac{x-y}{|x-y|}} \frac{(u(x)-u(y))(\varphi(x)-\varphi(y))}{|x-y|^{n+2s}}\, dx\, dy = g[\varphi] \quad \forall \varphi \in C_c^\infty(\R^n). 
\]
Then for any $r \in (0,\min\{t, \alpha\})$, $\laps{r} u \in W^{t-r,2}(\R^n)$ satisfies
\begin{equation}\label{eq:highergoal}
\begin{split}
\int_{\R^n} \int_{\R^n}  &K\brac{x,|x-y|,\frac{x-y}{|x-y|}} \frac{(\laps{r} u(x)-\laps{r} u(y))(\varphi(x)-\varphi(y))}{|x-y|^{n+2s}}\, dx\, dh \\
&= g[\laps{r} \varphi] + \mathcal{H}(u,\varphi) \quad \forall \varphi \in C_c^\infty(\R^n), 
\end{split}
\end{equation}
where $\mathcal{H}(u,\varphi)$ is a bilinear form. 
Moreover, we have the following estimate for any for any $x_{0}\in \mathbb{R}^{n}$, $\eps > 0$,  and $R>0$ 
\begin{equation}\label{eq:highergoal:est}
 \mathcal{H}(u,\varphi) \aleq C(\eps, R)\, [u]_{W^{t,p}(\R^n)} (\|\varphi\|_{L^{p'}(\mathbb{R}^n)} + \|\varphi\|_{W^{2s-t+\eps,p'}(\R^n)} 
\end{equation}
for any $\varphi\in C_c^\infty(B(x_0, R))$.

\end{theorem}
\begin{proof}
Let $r \in (0,t-s)$. Then $\laps{r} u \in W^{s,2}(\R^n)$, and we have for  any $\varphi \in C_c^\infty(\R^n)$, 
\[
\begin{split}
 &\int_{\R^n} \int_{\R^n}  K\brac{x,|x-y|,\frac{x-y}{|x-y|}} \frac{(\laps{r}u(x)-\laps{r}u(y))(\varphi(x)-\varphi(y))}{|x-y|^{n+2s}}\, dx\, dy \\
 =&\int_{\R^n} \int_{\R^n}  K\brac{x,|h|,\frac{h}{|h|}} \frac{(\laps{r}_x u(x)-\laps{r}_xu(x+h))(\varphi(x)-\varphi(x+h))}{|h|^{n+2s}}\, dx\, dh  \\
=& \int_{\R^n} \int_{\R^n}  \frac{(u(x)-u(x+h))(\laps{r}_x \brac{K\brac{x,|h|,\frac{h}{|h|}} \varphi(x)}-\laps{r}_x \brac{K\brac{x,|h|,\frac{h}{|h|}} \varphi(x+h))}}{|h|^{n+2s}}\, dx\, dh  \\
  =&\int_{\R^n} \int_{\R^n}  K\brac{x,|x-y|,\frac{x-y}{|x-y|}} \frac{(u(x)-u(y))(\laps{r}\varphi(x)-\laps{r}\varphi(y))}{|x-y|^{n+2s}}\, dx\, dh \\
 &+\int_{\R^n} \int_{\R^n}  \frac{(u(x)-u(x+h))\, \kappa(x,|h|,\frac{h}{|h|})}{|h|^{n+2s}}\, dx\, dh  ,
 \end{split} 
\]
where 
\[
\kappa(x,\rho,\frac{h}{|h|}) = c\int_{\R^n} \frac{\brac{K\brac{x,\rho,\frac{h}{|h|}} - K\brac{z,\rho,\frac{h}{|h|}}} \brac{\varphi(z)-\varphi(z+h)}}{|x-z|^{n+r}}\, dz.
\]
Here we have applied the elementary formula
\[
 \laps{r} (ab)(x) - a(x) \laps{r} b(x) = c \int_{\R^n}  \frac{(a(x)-a(z))b(z)}{|x-z|^{n+r}}\, dz. 
\]

So we have shown \eqref{eq:highergoal} for
\[
\begin{split}
 \mathcal{H}(u,\varphi)  :=& \int_{\R^n} \int_{\R^n}  \frac{(u(x)-u(y))\, \kappa(x,|x-y|,\frac{x-y}{|x-y|})}{|x-y|^{n+2s}}\, dx\, dy.
 \end{split}
\]
For a given $\epsilon > 0$ and $R>0$,  it remains to prove the estimate \eqref{eq:highergoal:est} for $\mathcal{H}(u,\varphi)$ for any $\varphi\in C_c^\infty(B(x_0, R))$.  To that end, first by H\"older's inequality,
\[
 |\mathcal{H}(u,\varphi)| \leq [u]_{W^{t,p}(\R^n)} \brac{\int_{\R^n} \int_{\R^n} \frac{|\kappa\brac{x,|x-y|,\frac{x-y}{|x-y|}}|^{p'}}{|x-y|^{n+(2s-t)p'}}\, dx\, dy}^{\frac{1}{p'}}.
\]

Recall (cf. \cite[Proposition 6.6.]{S18Arma}) that for any $\sigma \in (0,1)$,
\[
 |f(z)-f(z+h)| \aleq |h|^\sigma \brac{\mathcal{M}\laps{\sigma} f(z)+\mathcal{M}\laps{\sigma} f(z+h)},
\]
where $\mathcal{M}$ is a finite power of the maximal function. Then, by boundedness and the $\alpha$-H\"older continuity of $K$, \eqref{eq:Khoelder}, for any $\beta \in (r,\alpha)$ and  $\sigma\in (0, 1)$ we have 
\begin{equation}\label{eq:kappest23423}
\begin{split}
 &|\kappa(x,|x-y|,\frac{x-y}{|x-y|})| \\
 \aleq&|x-y|^\sigma\, \int_{\R^n} |x-z|^{\beta-r-n}\,  \brac{\mathcal{M}\laps{\sigma}\varphi(z) + \mathcal{M}\laps{\sigma}\varphi(z+y-x)}dz\\
  \aeq&|x-y|^\sigma \brac{\lapms{\beta-r} \mathcal{M} \laps{\sigma} \varphi(x) + \lapms{\beta-r} \mathcal{M} \laps{\sigma} \varphi(y)}.
 \end{split}
\end{equation}
Now writing 
\[
\begin{split}
&\int_{\R^n} \int_{\R^n} \frac{|\kappa\brac{x,|x-y|,\frac{x-y}{|x-y|}}|^{p'}}{|x-y|^{n+(2s-t)p'}}\, dx\, dy\\
 &\leq   \int_{\R^n} \int_{|x-y| \leq 1} \frac{|\kappa\brac{x,|x-y|,\frac{x-y}{|x-y|}}|^{p'}}{|x-y|^{n+(2s-t)p'}}\, dx\, dy + \int_{\R^n} \int_{|x-y| \geq 1} \frac{|\kappa\brac{x,|x-y|,\frac{x-y}{|x-y|}}|^{p'}}{|x-y|^{n+(2s-t)p'}}\, dx\, dy\\
 &:=I_1 + I_2.
\end{split}
\]
We estimate $I_1$ and $I_2$ separately. 
We estimate $I_1$ first. Choose $\epsilon>0$ such that $2s-t+\eps <1$. This choice is possible because $t>s$.
Now for any $\sigma_1 \in (2s-t, 2s-t+\eps)$  
we apply  \eqref{eq:kappest23423} to obtain that 
\[
\begin{split}
I_1&\aleq  \int_{\R^n} \int_{|x-y| \leq 1} {\frac{|\lapms{\beta-r} \mathcal{M} \laps{\sigma_1} \varphi(x) |^{p'}}{|x-y|^{n+(2s-t-\sigma_1)p'}}} dxdy+  \int_{\R^n} \int_{|x-y| \leq 1} {\frac{|\lapms{\beta-r} \mathcal{M} \laps{\sigma_1} \varphi(y) |^{p'}}{|x-y|^{n+(2s-t-\sigma_1)p'}}} dxdy\\
&\aleq \int_{\R^n}  {|\lapms{\beta-r} \mathcal{M} \laps{\sigma_1} \varphi(x) |^{p'}}dx,
\end{split}
\]
where use the fact that $\int_{|h|\leq 1}{\frac{1}{|h|^{n+(2s-t-\sigma_1)p'} }} dh <\infty,$ since $2s-t-\sigma_1 <0.$  Similarly, to estimate $I_2$  we apply  \eqref{eq:kappest23423} for $\sigma_{2}\in (0, 2s-t)$ and use the fact that  $\int_{|h|\geq 1}{\frac{1}{|h|^{n+(2s-t-\sigma_2)p'} }} dh <\infty$  to obtain that  
\[
I_{2} \aleq \int_{\R^n}  {|\lapms{\beta-r} \mathcal{M} \laps{\sigma_2} \varphi(x) |^{p'}}dx. 
\]
 Thus we have shown that 
\[
\brac{\int_{\R^n} \int_{\R^n} \frac{|\kappa\brac{x,|x-y|,\frac{x-y}{|x-y|}}|^{p'}}{|x-y|^{n+(2s-t)p'}}\, dx\, dy}^{1/p'} \\
 \aleq \|\lapms{\beta-r} \mathcal{M} \laps{\sigma_1} \varphi\|_{L^{p'}(\R^n)} + \|\lapms{\beta-r} \mathcal{M} \laps{\sigma_2} \varphi\|_{L^{p'}(\R^n)}.
\]

{Now we use the the continuous embedding $H^{t, p}(\mathbb{R}^{n})\subset H^{s,p}(\mathbb{R}^n)$ for $0<s <t$ and $1<p<\infty$, see  \cite{RS96},  with the inequality 
$[f]_{H^{s,p}} \aleq \|f\|_{L^p} + [f]_{H^{t,p}}$, for any given $\delta>0$ small, we may choose $\beta>r$ sufficiently close to $r$ so that 
\[
\begin{split}
\brac{\int_{\R^n} \int_{\R^n} \frac{|\kappa\brac{x,|x-y|,\frac{x-y}{|x-y|}}|^{p'}}{|x-y|^{n+(2s-t)p'}}\, dx\, dy}^{1/p'} \aleq \|\varphi\|_{L^{p'-\delta}(\R^n)} + \|\laps{\sigma_1} \varphi\|_{L^{p'-\delta}(\R^n)}.
\end{split}
\]
After noting that $\varphi$ is supported on $B(x_0,R)$, we have that $ \|\varphi\|_{L^{p'-\delta}(\R^n)} \aleq  \|\varphi\|_{L^{p'}(\R^n)}$. Also using \cite[Lemma 2.3]{MSYHoelder20} and the usual disjoint support argument we have 
\[
\|\laps{\sigma_1} \varphi\|_{L^{p'-\delta}(\R^n)} \aleq \| \varphi\|_{L^{p'}(\R^n)} + \|\laps{\sigma_1} \varphi\|_{L^{p'}(\R^n)}
\]
Since $\sigma_1 < 2s-t+\eps$ we can use yet again Sobolev embedding to obtain (here the constant depends on $\eps > 0$)
\[
\brac{\int_{\R^n} \int_{\R^n} \frac{|\kappa\brac{x,|x-y|,\frac{x-y}{|x-y|}}|^{p'}}{|x-y|^{n+(2s-t)p'}}\, dx\, dy}^{1/p'} 
 \aleq \|\varphi\|_{L^{p'}(\R^n)} + [\varphi]_{W^{2s-t+\eps,p'}(\R^n)}.
\]}
This proves estimate \eqref{eq:highergoal:est}.
\end{proof}

\section{Proof of main results}
\begin{proof}[Proof of \Cref{th:main}]
The proof uses an  iterative argument similar to the one used in \cite{SYCZ20}. Here we present a sketch of the argument.  Let $k \in \N$, and choose bounded sets $\Omega' \subset \Omega_k \subset \Omega_{k-1} \subset \ldots \Omega_1 \subset \Omega=B(x_0, \mathfrak{R})$. Cover $\Omega_1$ by finitely many balls $B(x_i, 5R)$ where $R$ is chosen such that $B(x_i, 60\sqrt{n}R)\subsubset \Omega$, $i=1, 2, \cdots, l$.  
Applying \Cref{pr:rewrite} and then  \Cref{th:fixedpoint} we obtain that $u \in W^{s_1,p_1}(B(x_i, 5R))$ (since $u = v$ in $B(x_i, 5R)$), and thus $u \in W^{s_1,p_1}(\Omega_1)$, for some $s_1 \geq s$ and $p_1\geq p$. Repeat this procedure on $\Omega_2$ and we find for some $s_2 \geq s_1$ and $p_2 \geq p_1$ that $u \in W^{s_2,p_2}(\Omega_2)$. After $k$ steps we have obtained $u \in W^{s_k,p_k}(\Omega')$. We can choose $s_k$ and $p_k$ so that \Cref{pr:rewrite} and \Cref{th:fixedpoint} are applicable in each step and after $k$ steps we have $p_k = p$ and $s_k = s$, and thus the theorem is proven.
\end{proof}

\begin{proof}[Proof of \Cref{th:main:above1}]
We may assume that $2s > 1$, otherwise there is nothing to show. By the \Cref{th:main} we have already obtained a $W^{t,p}_{loc}$-estimate for any $t < 1$. Let $r : = t-s$ for $t < 1$, $t \approx 1$ then applying \Cref{th:diffeq}, we obtain that $\laps{r} u$ satisfies an equation to which we can apply (locally) \Cref{th:main}. In this way we can keep bootstrapping to the claim.
\end{proof}

\begin{proof}[Proof of \Cref{co:main:diffeo}]
Set 
$
 K(x,r,h) := \brac{\frac{r}{|\phi(x)-\phi(x+rh)|}}^{n+2s}
$
and recall that $\phi$ bi-Lipschitz, that is
\begin{equation}\label{eq:bilipest}
 0 < \inf_{x \in \Omega, r \geq 0, |h| =1} \frac{|\phi(x)-\phi(x+rh)|}{|x-y|} < \sup_{x \in \Omega, r \geq 0, |h| =1} \frac{|\phi(x)-\phi(x+rh)|}{|x-y|}   < \infty.
\end{equation}
We conclude that $K$ is bounded from below (by a positive number) and above. Also $K$ is differentiable in $C^\alpha$ with respect to $x$. Indeed
\[
\begin{split}
 &\abs{\frac{r}{|\phi(x)-\phi(x+rh)|} - \frac{r}{|\phi(y)-\phi(y+rh)|}}\\
 \leq&\frac{r}{|\phi(x)-\phi(x+rh)|\, |\phi(y)-\phi(y+rh)|} \abs{|\phi(y)-\phi(y+rh)|-|\phi(x)-\phi(x+rh)|}\\
 \leq&\frac{r}{|\phi(x)-\phi(x+rh)|\, |\phi(y)-\phi(y+rh)|} \abs{\phi(y)-\phi(y+rh)-(\phi(x)-\phi(x+rh))}.
 \end{split}
\]
By the fundamental theorem of calculus,
\[
\begin{split}
 &\abs{\phi(y)-\phi(y+rh)-\phi(x)-\phi(x+rh)} \\
 \leq &|r| \int_{0}^1 \abs{D\phi(y+trh)-D\phi(x+trh)}\, dt\\
 \aleq &|r| |x-y|^\alpha\\
 \end{split}
\]
Thus we have shown
\[
\begin{split}
 &\abs{\frac{r}{|\phi(x)-\phi(x+rh)|} - \frac{r}{|\phi(y)-\phi(y+rh)|}}\\
 &\aleq\frac{r^2}{|\phi(x)-\phi(x+rh)|\, |\phi(y)-\phi(y+rh)|} |x-y|^\alpha\\
&\aleq |x-y|^\alpha,
\end{split}
\]
where the last inequality follows from  \eqref{eq:bilipest}. 
Since $K(x,r,h)$ is a positive power of this expression and $K(x,r,h)$ is bounded, we also get the $C^\alpha$-H\"older continuity of $K$.

Now the claim follows from \Cref{th:main} and \Cref{th:main:above1} if necessary.
\end{proof}

\section{Applications to the p-Laplacian: Proof of Corollary~\ref{co:plap}}\label{s:co:plap}
The main idea of \Cref{co:plap} is that locally around a point $x_0$,
\[
 |u(x)-u(y)|^{p-2} \aeq |x-y|^{p-2} \abs{\nabla u(x_0) \frac{x-y}{|x-y|}}^{p-2}.
\]
The kernel $K(h) := |\nabla u(x_0) \frac{h}{|h|}|^{p-2}$ satisfies the ellipticity assumptions as in \eqref{eq:lambdaangleass}.

So one expects that we can rewrite \eqref{eq:plapeq} into 
\[
 \int_{\Omega} \int_{\Omega} \frac{K(\frac{x-y}{|x-y|}) (u(x)-u(y))\, (\varphi(x)-\varphi(y))}{|x-y|^{n+sp-p+2}}\, dx\, dy = \mathcal{F}(u,\varphi) + \int_{\Omega} f \varphi
\]
where $\mathcal{F}$ is a small distortion. Since by assumption $sp-p+2 > 0$ this is still a nonlocal differential equation of the type treated in our paper. Differentiating both sides, we expect an equation of the form
\[
 \int_{\Omega} \int_{\Omega} \frac{K(\frac{x-y}{|x-y|}) (\partial_\alpha u(x)-\partial_\alpha u(y))\, (\varphi(x)-\varphi(y))}{|x-y|^{n+sp-p+2}}\, dx\, dy = \mathcal{F}(\partial_\alpha u,\varphi) + \int_{\Omega} \partial_\alpha f \varphi
\]
To make this more precise we apply discrete differentiation. Set
\[
 \delta_\tau f(x) := f(x+\tau) - f(x).
\]
Then we have the following result which obtains an equation for $\delta_\tau u$ to which \Cref{th:main} is applicable.

\begin{proposition}\label{pr:plaprew}
Let $s > 1-\frac{2}{p}$, $p \geq 2$, $s \in (0,1)$, $\Omega \subseteq \R^n$ open. Let $u \in L^\infty(\R^n) \cap W^{s,p}(\R^n)$ be a solution to \eqref{eq:plapeq}.

Assume that moreover $u \in W^{\frac{sp-p+2}{2},2}(B(x_0,60\sqrt{n}R)) \cap C^{1}(B(x_0,60\sqrt{n}R))$ where the ball $B(x_0,60\sqrt{n}R) \subset \Omega$.

For each $\tau  \in B(0,R)$ there exist $K_\tau(x,r,h)$ that satisfies the conditions of \Cref{th:main}, and  for all $\psi \in C_c^\infty\brac{B(x_0,\frac{1}{2} R)}, 
$
\begin{equation}\label{eq:plap:newpde}
 \int_{\Omega}\int_{\Omega} K_\tau(x,y) \brac{\delta_{\tau}u(x)-\delta_{\tau}u(y)}\brac{\psi(x)-\psi(y)} dx\, dy = \mathcal{G}_\tau(u,\psi) + \mathcal{H}_\tau(f,\psi) \quad\end{equation}
where $\mathcal{F}$ and $\mathcal{G}$ and $\mathcal{H}$ are bilinear and have the following properties: for any  $\beta > 1$ 
\begin{equation}\label{eq:plap:hest}
 |\mathcal{H}_\tau(f,\psi)| \leq C(\beta,\eta,\tilde{\eta}) |\tau| \|f\|_{C^{\beta}}\, \|\psi\|_{L^1} 
\end{equation}
and for any $q \geq p$,
\[
 |\mathcal{G}_{\tau}(u,\psi)|\aleq \|u\|_{L^\infty(\Omega)}^{p-2}\, \brac{|\tau| + \|\delta_\tau u\|_{L^{q}(B(x_0,R))}+\|\delta_\tau u\|_{L^{2}(\R^n)}} \|\psi\|_{L^{q'}(B(x_0,R))} 
\]
Moreover we have that if $u \in C^{1,\gamma}(B(x_0, R))$ then $K_\tau$ can be chosen in $C^\gamma(\Omega)$.
\end{proposition}
\begin{proof}
In the following we center all balls at $x_0$, i.e. we write $B(R)$ instead of $B(x_0,R)$. 

Let $\psi \in C_c^\infty(B(\lambda R))$ and set 
\[
 \mathcal{H}_\tau(f,\psi) := \int_{\R^n} \psi  \delta_\tau f.
\]
Observe that $\psi$ is localizing this so that the integration domain lies well within $\Omega$, and thus we have \eqref{eq:plap:hest}
\[
 |\mathcal{H}_\tau(f,\psi)| \leq C(\beta,\eta,\tilde{\eta}) |\tau| \|f\|_{C^{\beta}}\, \|\psi\|_{L^1(\Omega)}.
\]
which readily implies \eqref{eq:plap:hest}.

We begin now our computations by observing that from \eqref{eq:plapeq}
\[
\begin{split}
& \int_{\Omega}\int_{\Omega} \frac{|u(x)-u(y)|^{p-2}(u(x)-u(y)) 
\brac{\delta_{-\tau} \psi(x)-\delta_{-\tau} \psi(y)}}{|x-y|^{n+sp}}\, dx\, dy\\
 =& \mathcal{H}_\tau(f,\psi).
 \end{split}
\]
Next we have by the discrete integration by parts of $\delta_\tau$,
\[
 \begin{split}
& \int_{\Omega}\int_{\Omega} \frac{\delta_\tau \brac{|u(x)-u(y)|^{p-2}(u(x)-u(y))} \brac{ \psi(x)- \psi(y)}}{|x-y|^{n+sp}}\, dx\, dy\\
=& \mathcal{G}_{1, \tau}(u,\psi)+ \mathcal{H}_\tau(f,\psi).
\end{split}
\]
where with $D=[\Omega\times \Omega] \lap [(\Omega -\tau) \times (\Omega-\tau)]$, 
\[
 \mathcal{G}_{1,\tau}(u,\psi) = \iint_{D} \frac{\brac{|u(x+\tau)-u(y+\tau)|^{p-2}(u(x+\tau)-u(y+\tau))} \brac{ \psi(x)- \psi(y)}}{|x-y|^{n+sp}}\, dx\, dy\\
\]
We can now define $K_\tau$. Set 
\[
 \tilde{K}_\tau(x,r,h) := r^{2-p}(p-2)\int_0^1 \abs{t \brac{u(x)-u(x+rh)} + (1-t) \brac{u(x+\tau)-u(x+\tau+rh)}}^{p-2}
\]
We then have by the fundamental theorem of calculus (and since for $g(a) = |a|^{p-2}a$ we have $g'(a) = (p-1)|a|^{p-2}$)
\[
 \delta_\tau \brac{|u(x)-u(y)|^{p-2}(u(x)-u(y))} = |x-y|^{p-2} \tilde{K}_\tau\brac{x,|x-y|,\frac{x-y}{|x-y|}} \brac{\delta_\tau u(x) - \delta_\tau u(y)}
\]
Thus we have shown 
\[
 \begin{split}
&\int_{\Omega}\int_{\Omega} \frac{\tilde{K}_\tau\brac{x,|x-y|,\frac{x-y}{|x-y|}} (\delta_\tau u(x)-\delta_\tau u(y)) \brac{ \psi(x)- \psi(y)}}{|x-y|^{n+sp+2-p}}\, dx\, dy\\
=& \mathcal{G}_{1,\tau}(u,\psi)+ \mathcal{H}_\tau(f,\psi).
\end{split}
\]
Observe that for $x \in B(R)$ and $|r| \leq R$ and $|h| =1$ we have by the Lipschitz continuity of $u$ in $B(0,10R)$ that 
\[
 | \tilde{K}(x,r,h) | \leq C
\]
Moreover, since $u$ is continuously differentiable in $B(R)$,  for each $x \in B(x_0,\lambda R)$, $\lambda$, $|\tau|$ $r$ and $h$ small,
\[
\begin{split}
 \tilde{K}_\tau(x,r,h) =& (p-2)\int_0^1 \abs{t \brac{Du(x)+o_{r}(1)}h + (1-t) \brac{Du(x+\tau)+o_r(1)}h }^{p-2}\\
 =& (p-2)\int_0^1 \abs{t \brac{Du(x_0)+o_{\lambda}(1)+o_{r}(1)}h + (1-t) \brac{Du(x_0)+o_{\lambda+|\tau|}+o_r(1)}h }^{p-2}\\
 \end{split}
\]
Since $p-2 > 0$ (if $p=2$ there is nothing to show),
\[
 |\tilde{K}_\tau(x,r,h)-\tilde{K}_\tau(x_0,0,h)| \aleq o_r(1) + o_{\lambda+|\tau|}(1),
\]
so we have the required continuity from \Cref{th:main}.

Next, we show positivity of $\tilde{K}_\tau(x_0,0,h)$. Observe,
\[
\begin{split}
 \tilde{K}_\tau(x_0,0,h) =& (p-2)\int_0^1 \abs{Du(x_0)\cdot h + (1-t) o_{|\tau|}(1) h}^{p-2}\\
 \end{split}
\]
Let $v:= |\nabla u(x_0)|^{-1} \nabla u(x_0) \in \S^{n-1}$ (by assumption $|\nabla u(x_0)| > 0$).  Then, we can take $\Sigma$ a small neighbourhood of $v$, so that 
\[
 \inf_{h\in\Sigma} |\langle h,v \rangle | \geq \sigma >0
\]
for a given small $\sigma$ given. 
For each $h \in \Sigma$ let $t_0 > 0$ such that $ |\nabla u(x_0)|^{-1}|(1-t_0)o_{|\tau|}(1)| \leq \frac{\sigma}{2}$. Then
\[
|\nabla u(x_0)|^{-1}  \abs{Du(x_0)\cdot h + (1-t) o_{|\tau|}(1) h} \geq \frac{\sigma}{2} \quad \forall t \in [t_0,1].
\]
Thus for each $h \in \Sigma $,
\[
\begin{split}
 \tilde{K}_\tau(x_0,0,h) \geq & (p-2)\int_{t_0}^1 \abs{Du(x_0)\cdot h + (1-t) o_{|\tau|}(1) h}^{p-2} \geq C > 0.
 \end{split}
\]
That is \eqref{eq:lambdaangleass} is satisfied. It now easy to extend $\tilde{K}_\tau$ to a kernel that satisfies the conditions of \Cref{th:main} in all of $\Omega$, and we call this kernel $K_\tau$.
\[
 \begin{split}
&\int_{\Omega}\int_{\Omega} \frac{K_\tau\brac{x,|x-y|,\frac{x-y}{|x-y|}} (\delta_\tau u(x)-\delta_\tau u(y)) \brac{ \psi(x)- \psi(y)}}{|x-y|^{n+sp+2-p}}\, dx\, dy\\
=& \mathcal{G}_{1,\tau}(u,\psi)+\mathcal{G}_{2,\tau}(u,\psi)+ \mathcal{H}_\tau(f,\psi).
\end{split}
\]
where
\[
 \mathcal{G}_{2, \tau}(u,\psi) = \int_{\Omega}\int_{\Omega} \frac{(\tilde{K}_\tau-K_\tau)\brac{x,|x-y|,\frac{x-y}{|x-y|}} (\delta_\tau u(x)-\delta_\tau u(y)) \brac{ \psi(x)- \psi(y)}}{|x-y|^{n+sp+2-p}}\, dx\, dy
\]
Next we estimate each of the terms in the right hand side. 

\underline{Estimate of $\mathcal{G}_{1,\tau}$}
Observe that $|\Omega \lap [\Omega - \tau]| \aleq |\tau|$, and so also $|[\Omega\times \Omega] \lap [(\Omega -\tau) \times (\Omega-\tau)]| \aleq |\tau|$. Also if $x,y \in [\Omega\times \Omega] \lap [(\Omega -\tau) \times (\Omega-\tau)]$ and either $x$ or $y$ $\in \supp \psi = B(x_0,\frac{1}{2} R)$ then $|x-y| \aeq |x-y|+1$ (with constant depending on $R$). So
\[
 \mathcal{G}_{1, \tau}(u,\psi) \aleq |\tau| \|u\|_{L^\infty}^{p-1} \|\psi\|_{L^1}.
\]

\underline{Estimate of $\mathcal{G}_2$}
{
Observe that since $\psi \in C_c^\infty(B(x_0,\frac{1}{2}R))$ we have that 
\[
 x,y \in \Omega: (\tilde{K}_\tau-K_\tau)\brac{x,|x-y|,\frac{x-y}{|x-y|}} \brac{ \psi(x)- \psi(y)} \neq 0
\]
necessarily implies that $|x-y| \ageq R$. So we have 
{\tiny \[
\begin{split}
& \mathcal{G}_{2,\tau} (u,\psi)\aleq \int_{\Omega} \int_{\Omega} \brac{1+|u(x)-u(y)|^{p-2}+|u(x+\tau)-u(y+\tau)|^{p-2}} |\delta_\tau u(x)-\delta_\tau u(y)|\, |\psi(x)-\psi(y)|\, \frac{dx\, dy}{1+|x-y|^{n+sp-p+2}}
\end{split}
\]}}
Arguing again with the support of $\psi$ we should get for any $q \geq p$,
\[
 \mathcal{G}_{2,\tau}(u,\psi) \aleq \|u\|_{L^\infty(B(R))}^{p-2}\, \brac{\|\delta_\tau u\|_{L^{q}(B(R))}+\|\delta_\tau u\|_{L^{2}(\R^n)}} \|\psi\|_{L^{q'}(B(R))} 
\]

\end{proof}

\begin{proof}[Proof of Theorem~\ref{co:plap}]
Set $\alpha := \min\{sp-p+2,1\}$.
Iteratively applying \Cref{th:main} to the equation in \Cref{pr:plaprew} on balls $B(R_i)$ with $R_{i+1} < R_{i}$, for any $q_i$ such that the right-hand side is small, for arbitarily small $\eps > 0$.

\[
[\delta_h u]_{W^{\alpha-\eps,q_{i}-\eps}(B(R_{i+1}))} \aleq \brac{|\tau| + \|\delta_\tau u\|_{L^{q_i}(B(R))}+\|\delta_\tau u\|_{L^{2}(\R^n)}} 
\]
We choose $q_1 = p$ and $q_{i+1} > q_{i}$ such that $q_{i+1} < \frac{nq_i}{n-\alpha q_i}$ ($q_{i+1} = \infty$ if $n-\alpha q_i < 0$) then by Sobolev inequality (choosing $\eps$ appropriately small)
\[
\|\delta_\tau u\|_{L^{q_{i+1}}(B(R_{i+1}))} \aleq \brac{|\tau| + \|\delta_\tau u\|_{L^{q_i}(B(R_i))}+\|\delta_\tau u\|_{L^{2}(\R^n)}} 
\]
After finitely many steps $N$ we reach $q_N = \infty$. Thus we have obtained
\[
 [\delta_\tau u]_{C^{\alpha-\eps}(B(R_{N+1})} \aleq \brac{|\tau| + \|\delta_\tau u\|_{L^{p}(B(R_i))}+\|\delta_\tau u\|_{L^{2}(\R^n)}}.
\]
Dividing by $|\tau|^{s}$ we find that 
\[
 [u]_{C^{s+\alpha-\eps}(B(R_{N+1}))} < \infty.
\]
By assumptions we have $s+\alpha > 1$, so we can bootstrap regularity with the help of \Cref{th:main:above1}, and finally obtain $u \in C^{s+\min\{sp-p+2,2\}-\eps}$.
\end{proof}

\appendix
\newcommand{\be}{\begin{equation}}
\newcommand{\ee}{\end{equation}}
\newcommand{\vp }{\varphi }
\renewcommand{\k}{\kappa}
\newcommand{\cL}{{\mathcal L}}
\renewcommand{\a }{\alpha }
\newcommand{\s }{\sigma }
\newcommand{\n }{\nabla }

\section{A remark on H\"older regularity for equations with coercive kernels}\label{s:hoeldercont}
We consider symmetric kernels $K:\R^n\times \R^n\to [0,\infty] $ measurable and satisfying
\begin{equation}\label{eq:K-upb}
0 \leq K(x,y)\leq \k \qquad\textrm{ for all $x,y\in \R^n$}. 
\end{equation}
Moreover, we assume that for all $\vp \in \dot{H}^s(\R^n)$, 
\begin{equation}\label{eq:K-coerci}
\begin{split}
\frac{1}{\k} [\vp ]_{H^s(\R^n)}^2&:=\frac{1}{\k}\iint_{\R^{2n}} \frac{(\vp(x)-\vp(y))^2}{|x-y|^{n+2s}} \, dxdy\\
&\leq 
\iint_{\R^{2n}}K(x,y)\, \frac{(\vp(x)-\vp(y))^2}{|x-y|^{n+2s}}  dxdy := [\vp ]_{H^s_K(\R^n)}^2
\end{split}
\end{equation}
This includes translation anisotropic  kernels  $K$ supported  on cones as in \eqref{eq:lambdaangleass}, see \eqref{eq:low-Cone-BR}.

We recall that 
\[
\cL_K u[ \varphi] := \iint_{\R^{2n}}  K(x,y) \frac{(u(y)-u(x))(\varphi(y)-\varphi(x))}{|x-y|^{n+2s}}\, dx\, dy.
\]
We denote by $ \cL^1_s$ the space of $L^{1}_{loc}$ functions $u$ such that $x\mapsto {\frac{u(x)}{1 + |x|^{n + 2s}}}$ is $L^{1}(\mathbb{R}^{n}). $  
Our aim is to re-prove the following theorem.
\begin{theorem}\label{th:L2-to-L-infty}
Let $f\in L^\infty(\R^n)$ and $u\in H^s_{loc}(\R^n)\cap \cL^1_s$ such that $\cL_K u= f$ in $B_2$. Then, there exist $\a,C>0$ such
$$
\|u\|_{C^\a(B_1)}\leq C(\|u\|_{L^2(B_2)}+ \|u\|_{ \cL^1_s  }+\|f\|_{L^\infty(\R^n)}).
$$
\end{theorem}
To prove this result, we need some classical   fundamental results: Caccioppoli inequality, a nonlinear Sobolev inequality and a Logarithmic lemma.
\begin{lemma}[Caccioppoli inequality]\label{lem-cacc}

Let  $v\in H^s (B(R))\cap \cL^1_s $ be nonnegative and $f\in L^\infty(B(R))$ satisfy  
\be\label{eq:Dsv-eq-f-Liou}
\cL_K  v\leq  f \qquad\textrm{ in $   \  B(R).$}
\ee
 
 Then there exists $c=c(n,s,\k)>0$ such that    for every $\vp \in C^\infty_c( B(R))$, we have

\begin{align*}
&\frac12  [\vp v]_{H^s(\R^n)}^2 \leq c \|\n \vp\|_{L^\infty(B(R))}^2 R^{2-2s}   \|v\|^2_{L^2(B(R))} \\
 &+ c\|\vp\|_{L^\infty(B(R))} \left[\int_{\R^n\setminus B(R)}v(y)  |x-y|^{-2s-n}\, dy+ \|f\|_{L^\infty(B(R))}\right] \|v \vp \|_{L^1(B(R))} .
\end{align*}
\end{lemma}
\begin{proof}
By direct computations, we get 
\begin{align*}
(v(x)-v(y))[v(x)\vp^2(x)-v(y)\vp^2(y)]&=  (\vp(x)v(x)-\vp(y)v(y))^2 -v(x)v(y)(\vp(x)-\vp(y))^2.
\end{align*}
Testing the equation \eqref{eq:Dsv-eq-f-Liou} with $v\vp^2$,   and using the identity above together with the symmetry of $K$,  we  get
\begin{align}\label{eq:before-cacciop}
\frac12 [\vp v]_{H^s_K(\R^n)}^2 &\leq \frac12 \iint_{\R^{2n}}v(x)v(y)(\vp(x)-\vp(y))^2 {\frac{K(x,y)}{|x-y|^{n+2s}}}\, dydx  +\int_{\R^n} f (  x) v(x)  \vp^2(x)\, dx. 
\end{align}
By   H\"older and Young's inequalities, we get 
\begin{align}
\int_{\R^n}| f (  x) |v(x)  \vp^2&(x)\, dx \leq    \|f\|_{L^\infty(B(R))} \|\vp \|_{L^\infty(B(R))} \|v\vp\|_{L^1(\R^n)}  .
\end{align} 
 Next, we write
 \begin{align}
 &\iint_{\R^{2n}}v(x)v(y)(\vp(x)-\vp(y))^2 {\frac{K(x,y)}{|x-y|^{n+2s}}}\, dydx\\
 & = \iint_{B(R)\times B(R)}v(x)v(y)(\vp(x)-\vp(y))^2 {\frac{K(x,y)}{|x-y|^{n+2s}}}\, dydx \nonumber \\
& + 2   \iint_{B(R)\times \R^n\setminus B(R)}v(x)v(y)(\vp(x)-\vp(y))^2 {\frac{K(x,y)}{|x-y|^{n+2s}}}\, dydx  .  \label{eq:2-Caccipp}
 \end{align}
 We put $\mu_1(x,y)=|x-y|^{-2s-n}$.
By \eqref{eq:K-upb},  H\"older and Young's inequalities, we get 
\begin{align}
 &\iint_{B(R)\times B(R)}v(x)v(y)(\vp(x)-\vp(y))^2 {\frac{K(x,y)}{|x-y|^{n+2s}}}\, dydx\\
 &\leq 4  \k \int_{B(R)}v^2(y)\int_{B(R)}(\vp(x)-\vp(y))^2 \mu_1(x,y)\, dydx \nonumber\\
 &\leq 4  \|\n \vp\|_{L^\infty(B(R))}^2\, \, \k \int_{B(R)}|v|^2(y)\int_{B(R)} |x-y|^2 \mu_1(x,y)\, dydx \nonumber\\
&\leq c \|\n \vp\|_{L^\infty(B(R))}^2R^{2-2s}   \int_{B(R)}v^2(x)\, dx.  \label{eq:3-Caccipp} 
%
\end{align}
{Here we used that $\int_{B(R)} |x-y|^{2-2s-n} dy \aeq R^{2-2s}$ whenever $x \in B(R)$}.

Moreover, we estimate
\begin{align}
  \iint_{B(R)\times \R^n\setminus B(R)}&v(x)v(y)(\vp(x)-\vp(y))^2 K(x,y)\, dydx \nonumber\\
   & \leq   \k \iint_{B(R)\times \R^n\setminus B(R)}v(x)v(y)\vp(x)^2 \mu_1(x,y)\, dydx \nonumber \\
  & \leq   2 \k\int_{B(R)}v(x)|\vp(x)| \, dx\, \sup_{x\in B(R)}|\vp(x)|\int_{\R^n\setminus B _R}v(y)  \mu_1(x,y)\, dy.  \label{eq:4-Caccipp}
\end{align}
 Using \eqref{eq:K-coerci},  \eqref{eq:4-Caccipp}, \eqref{eq:3-Caccipp} and \eqref{eq:2-Caccipp} we get the result.
\end{proof}
An immediate consequence of the above Caccioppoli inequality and  \eqref{eq:K-coerci}, which implies the (nonlinear) Sobolev inequality, we can carry over the De Giorgi iteration, \cite{K07},  to  get the following local $L^{\infty}$ estimate. 
\begin{theorem}\label{lem:L-infty-bnd-u-ok}
Let  $u\in H^s(B_{2r})\cap\cL^1_s$ and $f\in L^\infty(\R^n)$    satisfy
$$
\cL_K  u=f \qquad\text{ in $B_{2r}$.} 
$$
Then there exists  $C=C(n,s,\k)>0$ such that   , 
$$
\sup_{B_r}|u| \leq C  \left(  r^{-n/2}\| u\|_{L^{2} (B_{2r})}+ r^{2s} \int_{|y|\geq r/2}|u(y)||y|^{-n-2s}\, dy+r^{2s}\|f\|_{L^\infty(B_{2r})} \right).
$$
\end{theorem}
To obtain H\"older estimates, we need the following result.
\begin{lemma}[Logarithmic lemma]\label{lem:log-lem}
Let  $v\in H^s (B(R))\cap \cL^1_s $  and $f\in L^\infty(B(R))$ such that \eqref{eq:Dsv-eq-f-Liou} holds
 and $u\geq 0$ in $B(R)$. 
 Then there exists $c=c(n,s,\k)>0$ such that    for every $r\in (0,R/2), d>0$, we have
\begin{align*}
&\int_{B_r\times B_r}K(x,y)\left| \log \frac{u(x)+d}{u(y)+d} \right|^2\, dxdy\\
&\leq c  r^{n-2s} \left( \frac{r^{2s}}{R^{2s}} d^{-1} \int_{|y|\geq R/2}|v_-(y)||y|^{-n-2s}\, dy +r^{2s}  d^{-1} \|f\|_{L^\infty(B_{r})} + 1 \right).
\end{align*}
\end{lemma}
\begin{proof}
The proof uses precisely the argument in \cite{CKPpmin}[proof of Lemma 3.1], where only the upper bound of $K$ was used. Note that in \cite{CKPpmin} a right hand side $f$ was not considered, however since the proof uses $\frac{\vp^2}{u+d}$, with $\vp \in C^\infty_c(B_{2r})$,  as a test function, we can easily obtain the estimate of the term involving $f$.
\end{proof}
Note that in view of  \eqref{eq:K-coerci}, we have the following Poincar\'e  inequality,
\be 
\begin{split}
& c \int_{B_r} \left|u  -(u)_{B_r} \right|^2\, dx\\
&\leq   r^{2s-n}\int_{B_r\times B_r}(u(x)-u(y))^2K(x,y)\, dxdy \qquad\textrm{ for all $u\in H^s_{loc}(\R^n)$}.
\end{split}
\ee
In  view of this and Lemma \ref{lem:log-lem}, we have 
\begin{corollary}\label{cor:cor-log}
Under the assumptions of Lemma \ref{lem:log-lem}, we have 
\be
 c \int_{B_r} \left|w  -(w)_{B_r} \right|^2\, dx\leq   \left( \frac{r^{2s}}{R^{2s}} d^{-1} \int_{|y|\geq R/2}|v_-(y)||y|^{-n-2s}\, dy +r^{2s}d^{-1} \|f\|_{L^\infty(B_{r})} +1 \right),
\ee
where  $w=\min (( \log(a+d)-\log(u+d))_+,\log b) $ for all $a,d>0$ and $b>1$.
\end{corollary}
 
 The H\"older continuity is a consequence of Theorem \ref{th:L2-to-L-infty},   Corollary \ref{cor:cor-log}, Lemma \ref{lem-cacc} and the following growth estimates of the oscillation  of $u$.  
 \begin{lemma}
 Let $f\in L^\infty(\R^n)$ and $u\in H^s_{loc}(\R^n)\cap \cL^1_s$ such that $\cL_K u= f$ in $B_2$.  Then, there exists $\s,\a\in (0,1)$ depending only on $n,s,\k$ such that, for $i\in  \N$,   
 $$
 \sup_{B_{\s^{i}}}u- \inf_{B_{\s^{i}}}u\leq  \s^{\a i}\left(  \int_{|y|\geq 1/2}|u(y)||y|^{-n-2s}\, dy+ \|f\|_{L^\infty(B_{2})} +  \|u\|_{L^2(B_{2})}   \right ).
 $$
 \end{lemma} 
 The proof is exactly the same as the \cite[proof of Lemma 5.1]{CKPpmin}.  We skip the details.

\bibliographystyle{abbrv}
\bibliography{bib}

\end{document}